\documentclass{ios-book-article}

\usepackage[all]{xy}
\usepackage{amssymb}
\usepackage{algorithm}
\usepackage{algorithmic}
\usepackage{amsmath,amsbsy,amsfonts,amsopn,amstext}
\usepackage{euscript,amssymb,amsbsy,amsfonts,amsthm,latexsym,amsopn,amstext, amsxtra,euscript,amscd}
\usepackage{graphicx, verbatim}

\input xy \xyoption{all}

\newtheorem{thm}{Theorem}
\newtheorem{proposition}{Proposition}
\newtheorem{lem}{Lemma}
\newtheorem{remark}{Remark}
\newtheorem{cor}{Corollary}
\theoremstyle{definition}
\newtheorem{definition}{Definition}
\theoremstyle{remark}
\newtheorem{rem}{Remark}
\newtheorem{example}{Example}

\def\embd{\hookrightarrow}

\newcommand{\ch}[2]
{\begin{bmatrix}
 #1 \\
 #2\\
\end{bmatrix}}

\newcommand{\chr}[4]
{\begin{bmatrix}
 #1 & #2\\
 #3 & #4
\end{bmatrix}}

\def\Q{\mathbb Q}
\def\bP{\mathbb P}
\def\bC{\mathbb C}
\def\bZ{\mathbb Z}
\def\C{\mathbb C}
\def\Z{\mathbb Z}
\def\Aut{\mbox{Aut}}

\def\X{\mathcal C}
\def\p{\mathfrak p}

\def\M{\mathcal M}
\def\L{\mathcal L}
\def\T{\th}
\def\D{\Delta}
\def\O{\Omega}

\def\Hs{\mathcal H_\s}
\def\A{\mathcal A}
\def\D{\Delta}

\def\R{\mathbb R}

\def\A{{\bold A}}

\def\P{{\mathbb P}}
\def\O{{\mathcal O}}
\def\Pic{\text{\rm Pic}}

\def\Jac{\text{\rm Jac }}

\def\Im{\text{\rm Im}}

\def\H_2{{\mathcal H_2}}  
\def\M{{\mathcal M}}
\def\A{\mathcal A}

\def\L{\mathcal L}

\def\v{\mathfrak v}
\def\u{\mathfrak u}

\def\a{\alpha}

\def\t{\tau}
\def\s{\sigma}

\def\e{\eta}
\def\G{\Gamma}
\def\l{\lambda}

\def\iso{{\, \cong\, }}
\def\lar{\longrightarrow}

\def\s{\sigma}
\def\p{\mathfrak p}

\def\u{\mathfrak u}

\def\({\left(}
\def\){\right)}
\def\cO{{\mathcal O}}
\def\a{\alpha}

\def\g{\gamma}
\def\e{\varepsilon}

\def\l{\lambda}

\def\P{\mathcal P}
\def\r{\delta}  
\def\<{\langle}
\def\>{\rangle}

\def\sem{\rtimes }
\def\a{{\alpha}}
\def\th{{\theta}}

\def\sem{{\rtimes}}
\def\iso{{\, \cong\, }}
\def\e{{\xi}}
\def\r{{r}}

\def\s{\sigma}

\def\P{\mathcal P}
\def\<{\langle}
\def\>{\rangle}
\def\_u{{\mathfrak u}}
\def\J{J_{48}}

\def\u{{u}}
\def\v{{v}}

\def\det{\mbox{det }}

\def\J{\mbox{Jac }}

\newcommand{\K}{\mathcal{K}}
\newcommand{\mC}{\mathcal{C}}

\def\C{\mathbb C}
\def\H{\bold H}
\def\R{\mathbb R}
\def\Q{\mathbb Q}
\def\Z{\mathbb Z}

\def\A{{\bold A}}
\def\P{{\mathbb P}}
\def\O{{\mathcal O}}
\def\Pic{\text{\rm Pic}}

\def\Im{\text{\rm Im}}

\newcommand{\mE}{\mathcal{E}}

\def\G2#1{{\Gamma^{(2)}#1}}
\def\Gs2#1{{\Gamma_0^{(2)}#1}}
\def\Ys2#1{{Y_0^{(2)}#1}}

\begin{document}

\begin{frontmatter}          
%
\title{The arithmetic of genus two curves\thanks{Notes on three lectures given in the conference on New Challenges in Digital Communications in Rijeka, Croatia, May 31 - June 10, 2010.}
}

\runningtitle{The arithmetic of genus two}

\author{\fnms{T.} \snm{Shaska}\thanks{Corresponding Author: Tanush Shaska,
Department of Mathematics and Statistics, Oakland  University, Rochester Hills, MI, 48306,  USA; E-mail: shaska@oakland.edu}
}

\runningauthor{Shaska, Beshaj}

\address{Department of Mathematics, Oakland University}

\author{\fnms{L.} \snm{Beshaj}\thanks{The author wants to thanks the  Department of Mathematics and Statistics at Oakland  University for their hospitality during the time which this paper was written}}

\address{Department of Mathematics, University of Vlora.}

\maketitle

\begin{abstract}
Genus 2 curves have been an object of much mathematical interest since eighteenth century and continued interest to date. They have become  an important tool in many algorithms in cryptographic applications, such as factoring large numbers, hyperelliptic curve cryptography, etc. Choosing  genus 2 curves suitable for such applications is an important step of such algorithms. In existing algorithms often such curves are chosen using equations of moduli spaces of curves with decomposable Jacobians or Humbert surfaces.

In these lectures we will cover basic properties of genus 2 curves, moduli spaces of (n,n)-decomposable Jacobians and Humbert surfaces, modular polynomials of genus 2, Kummer surfaces, theta-functions and the arithmetic on the Jacobians of genus 2, and their applications to cryptography.   The lectures are intended for  graduate students in algebra, cryptography, and related areas.
\end{abstract}


\end{frontmatter}


\section{Introduction}
Genus 2 curves are an important tool in many algorithms in cryptographic applications, such as factoring large numbers, hyperelliptic curve cryptography, etc. Choosing such genus 2 curves is an important step of such algorithms.

Most of genus 2 cryptographic applications are based on the Kummer surface $\K_{a,b,c,d}$. Choosing small $a, b, c, d$ makes the arithmetic on the Jacobian very fast.  However,  only a small fraction of all choices of $a, b, c, d$ are   secure. We aren't able to recognize secure choices, because we aren't able to count points on such a large genus-2 Jacobian.

One of the techniques in counting such points explores genus 2 curves with decomposable Jacobians. All curves of genus 2 with decomposable Jacobians of a fixed level lie on a Humbert surface. Humbert surfaces of level $n=3, 5, 7$ are the only explicitly computed surfaces and are computed by the first author in \cite{Sh7}, \cite{deg3}, \cite{deg5}.

In these lectures we will cover basic properties of genus 2 curves, moduli spaces of $(n,n)$-decomposable Jacobians,  Humbert surfaces of discriminant $n^2$, modular polynomials of level $N$ for genus 2, Kummer surfaces, theta-functions, and the arithmetic on the Jacobians of genus 2.

Our goal is not to discuss genus 2 cryptosystems. Instead, this paper develops and describes mathematical methods which are used in such systems. In the second section, we discuss briefly invariants of binary sextics, which determine a coordinate on the moduli space $\M_2$. Furthermore, we list the groups that occur as automorphism groups of genus 2 curves.

In section three, we study the description of the locus of genus two curves with fixed automorphism group $G$. Such loci are given in terms of invariants of binary sextics. The stratification of the moduli space $\M_2$ is given in detail. A genus two curve $C$ with automorphism group of order > 4 usually has an elliptic involution. An exception from this rule is only the curve with automorphism group the cyclic group $C_{10}$. All genus two curves with elliptic involutions have a pair $(E, E^\prime)$ of degree 2 elliptic subcovers. We determine the $j$-invariants of such elliptic curves in terms of $C$. The space of genus 2 curves with elliptic involutions is an irreducible 2-dimensional sublocus $\L_2$ of $\M_2$ which is computed explicitly in terms of absolute invariants $i_1, i_2, i_3$ of genus 2 curves.  A birational parametrization of $\L_2$ is discovered by the first author in \cite{SV1} in terms of dihedral invariants $u$ and $v$. Such invariants have later been used by many authors in genus 2 cryptosystems.

In section four, we give a brief discussion of Jacobians of genus two curves. Such Jacobians are described in terms of the pair of polynomials $[\u(x), \v(x)]$ a la Mumford.
In section five, we discuss the Kummer surface. In the first part of this section we define 16 theta functions and the 4 fundamental theta functions. A description of all the loci of genus two curves with fixed automorphism group $G$ is given in terms of the theta functions. In detail this is first described in \cite{ShW} and \cite{PSW}

In section six, we study the genus two curves with decomposable Jacobians. These are the curves with degree $n$ elliptic subcovers. Their Jacobian is isogenous to a pair of degree $n$ elliptic subcovers $(E, E^\prime)$. For $n$ odd the space of genus two curves with $(n, n)$-split Jacobians correspond to the Humbert space of discriminant $n^2$. We state the main result for the case $n=3$ and give a graphical representation of the space. In each case the $j$-invariants of $E$ and $E^\prime$ are determined.

In section seven is given a brief description of the filed of moduli versus the field of definition problem.   Such problem is fully understood for genus 2 and is implemented in a Maple package in section eleven. In section eight, we study modular polynomials of genus 2. Some of the basic definitions are given and an algorithm suggested for computing such polynomials.  More details on these topic will appear in \cite{DS}. In section nine we focus on factoring large numbers using genus two curves. Such algorithm is faster than the elliptic curve algorithm. It is based on the fact that when genus two curves with split Jacobians are used the computations on the Jacobian are carried to the pair $(E, E^\prime)$ by reducing in half. We suggest genus two curves such that the Jacobian split in different ways. For example the Jacobian splits $(2,2)$, $(3, 3)$ and $(5, 5)$. Such curves have faster time that genus two curves determined up to now.

In the last section we describe a Maple package which does computation with genus 2 curves. Such package computes several invariants of genus two curves including the automorphism group, the Igusa invariants, the splitting of the Jacobian, the Kummer surface, etc.
These lectures will be suitable to the graduate students in algebra, cryptography, and related areas who need genus two curves in their research.

\bigskip

\noindent \textbf{Notation:}  Throughout this paper a genus two curve means a genus two irreducible algebraic curve defined over an algebraically closed field $k$. Such curve will be denoted by $C$ and its function field by $K=k(C)$.  The field of complex, rational, and real  numbers  will be denoted by $\C, \Q$, and $\R$ respectively. The Jacobian of $C$ will be denoted by $\Jac C$ and the Kummer surface by $\K (C)$ or simply $J_C, \K_C$.

\bigskip

\noindent \textbf{Acknowledgements:} The second author wants to thank the Department of Mathematics and Statistics at Oakland University for their hospitality during the time that this paper was written.

\section{Preliminaries on genus two curves}

Throughout this paper, let $k$ be an algebraically closed field of  characteristic zero and $C$ a genus 2 curve defined over $k$. Then $C$ can be described as a double cover of $\bP^1(k)$ ramified in 6 places $w_1, \dots , w_6$. This sets up a
bijection between isomorphism classes of genus 2 curves and unordered distinct 6-tuples $w_1, \dots , w_6 \in \bP^1 (k)$ modulo automorphisms of $\bP^1 (k) $. An unordered 6-tuple $\{w_i\}_{i=1}^6$ can be described by a binary sextic (i.e. a homogenous equation $f(X,Z)$ of degree 6).

\subsection{Invariants of binary forms}
In this section we define the action of $ GL_2(k)$ on binary forms and discuss the basic notions of their invariants.  Let $k[X,Z]$  be the  polynomial ring in  two variables and  let $V_d$ denote  the $(d+1)$-dimensional  subspace  of  $k[X,Z]$  consisting of homogeneous polynomials.
\begin{equation}  \label{eq1}
f(X,Z) = a_0X^d + a_1X^{d-1}Z + ... + a_dZ^d
\end{equation}
of  degree $d$. Elements  in $V_d$  are called  {\it binary  forms} of degree $d$. We let $GL_2(k)$ act as a group of automorphisms on $k[X, Z]$   as follows:
\begin{equation}
 M =
\begin{pmatrix} a &b \\  c & d
\end{pmatrix}
\in GL_2(k),   \textit{   then       }
\quad  M  \begin{pmatrix} X\\ Z \end{pmatrix} =
\begin{pmatrix} aX+bZ\\ cX+dZ \end{pmatrix}.
\end{equation}
This action of $GL_2(k)$  leaves $V_d$ invariant and acts irreducibly on $V_d$. Let $A_0$, $A_1$,  ... , $A_d$ be coordinate  functions on $V_d$. Then the coordinate  ring of $V_d$ can be  identified with $ k[A_0  , ... , A_d] $. For $I \in k[A_0, ... , A_d]$ and $M \in GL_2(k)$, define $I^M \in k[A_0, ... ,A_d]$ as follows
\begin{equation} \label{eq_I}
{I^M}(f):= I(M(f))
\end{equation}
for all $f \in V_d$. Then  $I^{MN} = (I^{M})^{N}$ and Eq.~(\ref{eq_I}) defines an action of $GL_2(k)$ on $k[A_0, ... ,A_d]$.
A homogeneous polynomial $I\in k[A_0, \dots , A_d, X, Z]$ is called a {\it covariant}  of index $s$ if
$$I^M(f)=\delta^s I(f)$$
where $\delta =\det(M)$.  The homogeneous degree in $A_1, \dots , A_n$ is called the {\it degree} of $I$,  and the homogeneous degree in $X, Z$ is called the {\it  order} of $I$.  A covariant of order zero is
called {\it invariant}.  An invariant is a $SL_2(k)$-invariant on $V_d$.

We will use the symbolic method of classical theory to construct covariants of binary forms. Let
\begin{equation}
\begin{split}
f(X,Z):= & \sum_{i=0}^n
\begin{pmatrix} n \\ i \end{pmatrix}a_i X^{n-i} \, Z^i, \\
 g(X,Z) := & \sum_{i=0}^m   \begin{pmatrix} m \\ i \end{pmatrix} b_i
 X^{n-i} \, Z^i    \\
\end{split}
\end{equation}
be binary forms  of  degree $n$ and $m$ respectively in $k[X, Z]$. We define the {\bf r-transvection}
\begin{equation}
(f,g)^r:= c_k \cdot \sum_{k=0}^r (-1)^k
\begin{pmatrix} r \\ k
\end{pmatrix} \cdot
\frac {\partial^r f} {\partial X^{r-k} \, \,  \partial Y^k} \cdot
\frac {\partial^r g} {\partial X^k  \, \, \partial Y^{r-k}}
\end{equation}
%
where $c_k=\frac {(m-r)! \, (n-r)!} {n! \, m!}$.  It is a homogeneous  polynomial in $k[X, Z]$ and therefore a covariant of order $m+n-2r$ and degree 2. In general, the $r$-transvection of two covariants of order $m, n$ (resp., degree $p, q$) is a covariant of order $m+n-2r$  (resp., degree $p+q$).

For the rest of this paper $F(X,Z)$ denotes a binary form of order
$d:=2g+2$ as below
\begin{equation}
F(X,Z) =   \sum_{i=0}^d  a_i X^i Z^{d-i} = \sum_{i=0}^d
\begin{pmatrix} n \\ i
\end{pmatrix}    b_i X^i Z^{n-i}
\end{equation}
where $b_i=\frac {(n-i)! \, \, i!} {n!} \cdot a_i$,  for $i=0, \dots , d$.  We denote invariants (resp., covariants) of binary forms by $I_s$ (resp., $J_s$) where the subscript $s$ denotes the degree (resp., the
order).

\begin{remark}
It is an open problem to determine the field of invariants of binary form of degree $d \geq 7$.
\end{remark}

\subsection{Moduli space of curves}
Let  $\M_2$ denote the moduli space of genus 2 curves. To describe $\M_2$ we need to find polynomial
functions of the coefficients of a binary sextic $f(X,Z)$ invariant under linear substitutions in $X,Z$ of
determinant one. These invariants were worked out by Clebsch and Bolza in the case of zero characteristic
and generalized by Igusa for any characteristic different from 2; see \cite{Bo}, \cite{Ig}, or \cite{SV1}
for a more modern treatment.

Consider a binary sextic, i.e. a homogeneous polynomial $f(X,Z)$ in $k[X,Z]$ of degree 6:
$$f(X,Z)=a_6 X^6+ a_5 X^5Z+ \dots  +a_0 Z^6.$$
\emph{Igusa  $J$-invariants} $\, \, \{ J_{2i} \}$ of $f(X,Z)$ are homogeneous polynomials of degree $2i$ in
$k[a_0, \dots , a_6]$, for $i=1,2,3,5$; see \cite{Ig}, \cite{SV1} for their definitions.
Here $J_{10}$ is simply the discriminant of $f(X,Z)$. It vanishes if and only if the binary sextic has a multiple linear factor. These
$J_{2i}$    are invariant under the natural action of $SL_2(k)$ on sextics. Dividing such an invariant by another one of the same degree gives an invariant under $GL_2(k)$ action.

Two genus  2 curves) in the standard form $Y^2=f(X,1)$ are isomorphic if and only if the
corresponding sextics are $GL_2(k)$ conjugate. Thus if $I$ is a $GL_2(k)$ invariant (resp., homogeneous $SL_2(k)$ invariant), then the expression $I(C)$ (resp., the condition $I(C)=0$) is well defined. Thus the $GL_2(k)$ invariants are functions on the  moduli space $\mathcal M_2$ of genus 2 curves. This $\mathcal M_2$ is an {\textbf affine variety with coordinate ring}
$$k[\mathcal M_2]=k[a_0, \dots , a_6, J_{10}^{-1}]^{GL_2(k)}$$
which is  the subring of degree 0 elements in $k[J_2, \dots ,J_{10}, J_{10}^{-1}]$. The \emph{absolute
invariants}
$$ i_1:=144 \frac {J_4} {J_2^2}, \, \,  i_2:=- 1728 \frac {J_2J_4-3J_6} {J_2^3},\, \,  i_3 :=486 \frac {J_{10}} {J_2^5}, $$
are even $GL_2(k)$-invariants. Two genus 2 curves with $J_2\neq 0$ are isomorphic if and only if they have
the same absolute invariants. If  $J_2=0 $ then we can define new invariants as in \cite{Sh2}. For the rest
of this paper if we say ``there is a genus 2 curve $C$ defined over $k$'' we will mean the $k$-isomorphism
class of $C$.

The   reason that the above invariants were defined with the $J_2$ in the denominator was so that their degrees (as rational functions in terms of $a_0, \dots , a_6$) be as low as possible. Hence, the computations in this case are simpler. While most of the computational results on \cite{Sh7}, \cite{deg3}, \cite{deg5} are expressed in terms of $i_1, i_2, i_3$ we have started to convert all the results in terms of the new invariants  
\[ t_1 = \frac {J_2^5} {J_{10}}, \quad  t_2 = \frac {J_4^5} {J_{10}^2}, \quad t_3  = \frac {J_6^5} {J_{10}^3}.\]

\subsection{Automorphisms of curves of genus two}

Let $\X$ be a genus 2 curve defined over an algebraically closed field $k$. We denote its automorphism group by  $\Aut(\X)=Aut(K/k)$ or similarly $\Aut (\X)$. In any characteristic different  from 2, the automorphism group $Aut(\X)$ is isomorphic to one of the groups given by the following lemma.
%
\begin{lem}\label{thm1} The  automorphism group $G$ of a
genus 2 curve $\X$ in characteristic $\ne 2$ is isomorphic to \ $C_2$, $C_{10}$, $V_4$, $D_8$, $D_{12}$, $C_3 \sem D_8$,
$ GL_2(3)$, or $2^+S_5$. The case $G \iso 2^+S_5$ occurs only in characteristic 5. If $G \iso \Z_3 \sem D_8$ (resp., $
GL_2(3)$), then $\X$ has equation $Y^2=X^6-1$ (resp., $Y^2=X(X^4-1)$). If $G \iso C_{10}$, then $\X$ has equation
$Y^2=X^6-X$.
\end{lem}
For the rest of this paper, we  assume that $char(k)=0.$



\section{Automorphism groups and the description of the corresponding loci.}
In this section we will study genus two curves which have and extra involution in the automorphism group. It turns out that there is only one automorphism group from the above lemma which does not have this property, namely the cyclic group $C_{10}$. However, there is only one genus two curve (up to isomorphism) which has automorphism group $C_{10}$. Hence, such case is not very interesting to us.

Thus, we will study genus two curves which have an extra involution, which is equivalent with having a degree 2 elliptic subcover; see the section on decomposable Jacobians for degree $n>2$ elliptic subcovers.

\subsection{Genus 2 curves with  degree 2 elliptic subcovers}

An \textbf{elliptic involution} of $K$ is  an  involution in $G$ which is different from $z_0$ (the
hyperelliptic involution). Thus the elliptic involutions of $G$ are in 1-1 correspondence with the elliptic
subfields of $K$ of degree 2 (by the Riemann-Hurwitz formula).

If $z_1$ is an elliptic involution and $z_0$ the hyperelliptic one, then $z_2:=z_0\, z_1$ is another elliptic
involution. So the elliptic involutions come naturally in pairs. This pairs also the elliptic subfields of
$K$ of degree 2. Two such subfields $E_1$ and $E_2$ are paired if and only if $E_1\cap k(X)=E_2\cap k(X)$.
$E_1$ and $E_2$ are $G$-conjugate unless $G\iso D_6$ or $G\iso V_4$.

\begin{thm} \label{mainthm_kap3}
Let $K$ be a genus 2 field and $e_2(K)$ the number of $Aut(K)$-classes of elliptic subfields of $K$ of degree 2. Suppose $e_2(K) \geq 1$. Then the classical invariants of $K$ satisfy the equation,
\begin{scriptsize}
\begin{equation}
\begin{split}\label{eq_L2_J}
-J_2^7 J_4^4+8748J_{10}J_2^4J_6^2507384000J_{10}^2J_4^2J_2-19245600J_{10}^2J_4J_2^3-592272J_{10}J_4^4J_2^2 \\
-81J_2^3J_6^4-3499200J_{10}J_2J_6^3+4743360J_{10}J_4^3J_2J_6-870912J_{10}J_4^2J_2^3J_6\\
+1332J_2^4J_4^4J_6-125971200000J_{10}^3 +384J_4^6J_6+41472J_{10}J_4^5+159J_4^6J_2^3\\
-47952J_2J_4J_6^4+104976000J_{10}^2J_2^2J_6-1728J_4^5J_2^2J_6+6048J_4^4J_2J_6^2+108J_2^4J_4J_6^3\\
+12J_2^6J_4^3J_6+29376J_2^2J_4^2J_6^3-8910J_2^3J_4^3J_6^2-2099520000J_{10}^2J_4J_6-236196J_{10}^2J_2^5\\
+31104J_6^5-6912J_4^3J_6^34 +972J_{10}J_2^6J_4^2 +77436J_{10}J_4^3J_2^4-78J_2^5J_4^5\\
+3090960J_{10}J_4J_2^2J_6^2-5832J_{10}J_2^5J_4J_6-80J_4^7J_2-54J_2^5J_4^2J_6^2-9331200J_{10}J_4^2J_6^2 = 0&\\
\end{split}
\end{equation}
\end{scriptsize}

\noindent Further, $e_2(K)=2$ unless $K=k(X,Y)$ with $$ Y^2=X^5-X  $$ in which case $e_2(K)=1$.
\end{thm}

\begin{lem} \label{lem1}
Suppose $z_1$ is an elliptic involution of $K$. Let $z_2=z_1z_0$, where $z_0$ is the hyperelliptic
involution. Let $E_i$ be the fixed field of $z_i$ for $i=1,2$. Then $K=k(X,Y)$ where
\begin{equation}
Y^2=X^6-s_1X^4+s_2X^2-1
\end{equation}
and $27-18s_1s_2-s_1^2s_2^2+4s_1^3+4s_2^3\neq 0$. Further $E_1$ and $E_2$ are the subfields $k(X^2,Y)$ and
$k(X^2, YX)$.
\end{lem}

We need to determine to what extent the normalization  above   determines the coordinate $X$. The
condition $z_1(X)=-X$ determines the coordinate $X$ up to a coordinate change by some $\g\in \Gamma$
centralizing $z_1$. Such $\g$ satisfies $\g(X)=mX$ or $\g(X) = \frac m X$, $m\in k\setminus \{0\}$. The
additional condition $abc=1$ forces $1=-\g(\a_1) \dots \g(a_6)$, hence $m^6=1$. So $X$ is determined up to a
coordinate change by the subgroup $H\iso D_6$ of $\Gamma$ generated by $\tau_1: X\to \e_6X$, $\tau_2: X\to
\frac 1 X$, where $\e_6$ is a primitive 6-th root of unity. Let $\e_3:=\e_6^2$. The coordinate change by
$\tau_1$ replaces $s_1$ by $\e_3s_2$ and $s_2$ by $\e_3^2 s_2$. The coordinate change by $\tau_2$ switches
$s_1$ and $s_2$. Invariants of this $H$-action are:
\begin{equation} u:=s_1 s_2, \quad v:=s_1^3+s_2^3 \end{equation}

\begin{remark}\label{remark1}
Such invariants were quite important in simplifying  computations for the locus $\L_2$. Later they have been used  by Duursma and Kiyavash to  show that genus 2 curves with extra involutions are suitable for the vector decomposition problem; see \cite{Du} for details. In this volume they are used again, see the paper by Cardona and  Quer. They were later generalized to higher genus hyperelliptic curves and were called \textbf{dihedral invariants}; see  \cite{GSh}.
\end{remark}

\begin{figure}[htbp]
\begin{center}
\includegraphics[width=8cm]{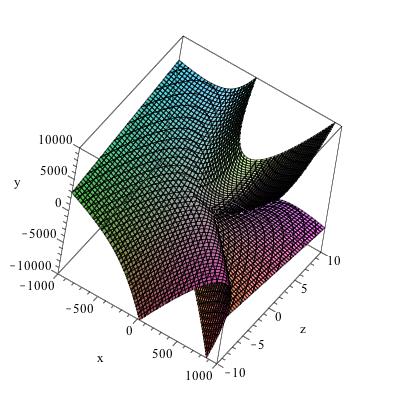}
\caption{The space $\L_2$ of genus 2 curves with extra involutions.}
\label{default}
\end{center}
\end{figure}

The following proposition determines the group $G$ in terms of $u$ and $v$.

\begin{proposition}\label{u_v}
Let $C$ be a genus 2 curve such that $G:=Aut(C)$ has an elliptic involution and $J_2\neq 0$. Then,

a) $G\iso \bZ_3 \sem D_4$ if and only if $(u,v)=(0,0)$ or $(u,v)=(225, 6750)$.

b) $G\iso W_1$ if and only if $u=25$ and $v=-250$.

c) $G\iso D_6$ if and only if $4v-u^2+110u-1125=0$, for $u\neq 9, 70 + 30\sqrt{5}, 25$.

Moreover, the classical invariants satisfy the equations,
\begin{scriptsize}
\begin{equation}
\begin{split}
-J_4J_2^4+12J_2^3J_6-52J_4^2J_2^2+80J_4^3+960J_2J_4J_6-3600J_6^2 &=0\\
864J_{10}J_2^5+3456000J_{10}J_4^2J_2-43200J_{10}J_4J_2^3-2332800000J_{10}^2-J_4^2J_2^6\\
-768J_4^4J_2^2+48J_4^3J_2^4+4096J_4^5 &=0\\
\end{split}
\end{equation}
\end{scriptsize}

d) $G\iso D_4$ if and only if $v^2-4u^3=0$, for $u \neq 1,9, 0, 25, 225$. Cases $u=0,225$ and $u=25$ are
reduced to cases a),and b) respectively. Moreover, the classical invariants satisfy \eqref{eq_L2_J} and the
following equation,
\begin{small}
\begin{equation}
\begin{split}
1706J_4^2J_2^2+2560J_4^3+27J_4J_2^4-81J_2^3J_6-14880J_2J_4J_ 6+28800J_6^2 &=0
\end{split}
\end{equation}
\end{small}
\end{proposition}

\begin{rem}
The following graphs are generated by Maple 13. Notice the singular point in both spaces of curves with automorphism group $D_4$ and $D_6$. Such points correspond to larger automorphism groups, namely the groups of order 24 and 48 respectively.   This can be easily seen from the group theory since $D_4 \embd \bZ_3 \sem D_4$ and  $D_6 \embd W_1$.
\end{rem}

\begin{figure}[htbp]
\begin{center}
\includegraphics[width=5cm]{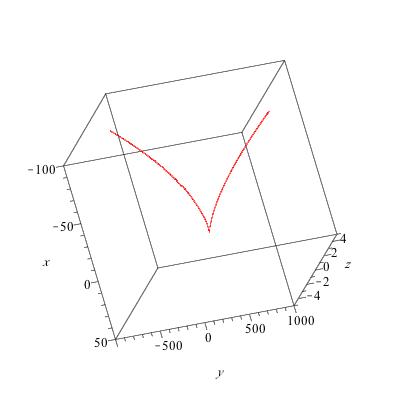} \quad \includegraphics[width=5cm]{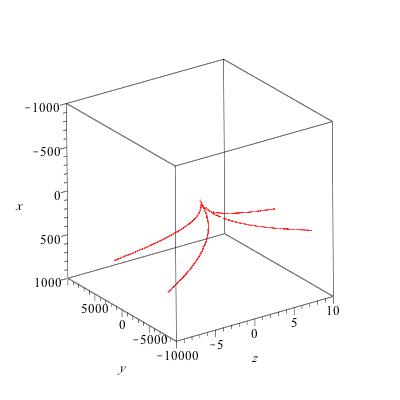}
\caption{The space of genus 2 curves with automorphism group $D_4$ and $D_6$ respectively.}
\label{default}
\end{center}
\end{figure}

\begin{proposition}\label{thm_3}  The mapping
$$A: (u,v) \lar (i_1, i_2, i_3)$$
gives a birational parametrization of $\L_2$. The fibers of $A$ of cardinality $>1$ correspond to those curves $C$
with $|\ Aut(C)| > 4$.
\end{proposition}

\begin{proof} See \cite{SV1} for the details.
\end{proof}

\subsubsection{Elliptic subcovers}

Let $j_1$ and $j_2$ denote the j-invariants of the elliptic curves $E_1$ and $E_2$ from Lemma~\ref{lem1}. The invariants $j_1$ and $j_2$  are the roots of the quadratic
\begin{equation}
\begin{split}
j^2+256\frac {(2u^3-54u^2+9uv-v^2+27v)} {(u^2+18u-4v-27)} j + 65536 \frac {(u^2+9u-3v)} {(u^2+18u-4v-27)^2} &=0 \label{j_eq}
\end{split}
\end{equation}

\subsubsection{Isomorphic elliptic subcovers}

The elliptic curves $E_1$ and $E_2$ are isomorphic when equation \eqref{j_eq} has a double root. The
discriminant of the quadratic is zero for
$$(v^2-4u^3)(v-9u+27)=0$$

\begin{remark} From lemma \ref{lem1},  $v^2=4u^3$ if and only if  $Aut(C)\iso
D_4$. So for $C$ such that $Aut(C)\iso D_4$, $E_1$ is isomorphic to $E_2$. It is easily checked that $z_1$
and $z_2=z_0 z_1$ are conjugate when $G\iso D_4$. So they fix isomorphic subfields.
\end{remark}
\par If $v=9(u-3)$ then the  locus of these curves is  given by,
\begin{equation}
\begin{split}
4i_1^5-9i_1^4+73728i_1^2i_3 -150994944 i_3^2=0 \\
289i_1^3 - 729 i_1^2 +54 i_1i_2 -i_2^2=0 \\
\end{split}
\end{equation}
For $(u,v)=(\frac 9 4, - \frac {27} 4)$ the curve has $Aut(C)\iso D_4$ and for $(u,v)=(137, 1206)$ it has
$Aut(C)\iso D_6$. All other curves with $v=9(u-3)$
 belong to the general case, so $Aut(C)\iso V_4$. The j-invariants of elliptic
curves are $j_1=j_2=256(9-u)$. Thus, these genus 2 curves are parameterized by the j-invariant of the
elliptic subcover.

\begin{remark}
This embeds the moduli space $\mathcal M_1$ into $\mathcal M_2$ in a functorial way.
\end{remark}

\subsection{Isogenous degree 2 elliptic subfields}

In this section we study pairs of degree 2 elliptic subfields of $K$ which are 2 or 3-isogenous. We denote by $\Phi_n(x,y)$ the n-th modular polynomial (see Blake et al. \cite{Blake} for the formal definitions. Two
elliptic curves with j-invariants $j_1$ and $j_2$ are $n$-isogenous if and only if $\Phi_n(j_1,j_2)=0$.
In the next section we will see how such modular polynomials can be generalized for higher genus.

\subsubsection{3-Isogeny.} Suppose $E_1$ and $E_2$ are 3-isogenous. Then, from equation (\ref{j_eq})
and $\Phi_3(j_1,j_2)=0$ we eliminate $j_1$ and $j_2$. Then,
\begin{equation}(4v-u^2+110u-1125)\cdot g_1(u,v)\cdot g_2(u,v)=0 \label{iso_3}\end{equation}
where $g_1$ and $g_2$ are given in \cite{SV1}.

Thus, there is a isogeny of degree 3 between $E_1$ and $E_2$ if and only if $u$ and $v$ satisfy equation
\eqref{iso_3}. The vanishing of the first factor is equivalent to $G\iso D_6$. So, if $Aut(C)\iso D_6$ then
$E_1$ and $E_2$ are isogenous of degree 3.

\subsubsection {2-Isogeny}
Below we give the modular 2-polynomial.
\begin{equation}
\begin{split}
\Phi_2 & =x^3-x^2y^2+y^3+1488xy(x+y)+40773375xy-162000(x^2-y^2)+\\
&8748000000(x+y)-157464000000000 \\
\end{split}
\end{equation}
Suppose $E_1$ and $E_2$ are isogenous of degree 2. Substituting $j_1$ and $j_2$ in $\Phi_2$ we get
\begin{equation}
 f_1(u,v)\cdot f_2(u,v)=0
\label{phi_2}
\end{equation}
where $f_1$ and $f_2$ are displayed in \cite{Sh}

\subsubsection{Other isogenies between  elliptic subcovers}

If $Aut(C)\iso D_4$, then $z_1$ and $z_2$ are in the same conjugacy class. There are again two conjugacy classes
of elliptic involutions in $Aut(C)$. Thus, there are two degree 2 elliptic subfields (up to isomorphism) of $K$.
One of them is determined by double root $j$ of the equation \eqref{j_eq}, for $v^2-4u^3=0$. Next, we
determine the j-invariant $j^\prime$ of the other degree 2 elliptic subfield and see how it is related to
$j$.
$$
\xymatrix{
 & & C  \ar@{->}[dl]  \ar@{->}[dll]   \ar@{->}[dr]  \ar@{->}[drr]    & & & \\
E_1  \ar@{~}[r]& E_2 & & E_1^\prime \ar@{~}[r]& E_2^\prime & \\
}
$$
If $v^2-4u^3=0$ then $Aut(C)\iso V_4$ and $\P=\{\pm 1, \pm \sqrt{a}, \pm \sqrt{b}\}$. Then, $s_1= a + \frac 1 a
+ 1=s_2$. Involutions of $C$ are $\tau_1: X\to -X$, $\tau_2: X\to \frac 1 X$, $\tau_3: X\to - \frac 1 X$.
Since $\tau_1 $ and $\tau_3$ fix no points of $\P$ then they lift to involutions in $Aut(C)$. They each determine a pair of isomorphic elliptic subfields. The j-invariant of elliptic subfield fixed by $\tau_1$ is the double
root of equation (\ref{j_eq}), namely
$$ j= -256 \frac {v^3} {v+1} $$
To find the j-invariant of the elliptic subfields fixed by $\tau_3$ we look at the degree 2 covering $\phi:
\bP^1\to \bP^1$, such that $\phi(\pm 1)=0$, $\phi(a)=\phi(-\frac 1 a)=1$, $\phi(-a)=\phi(\frac 1 a)=-1$, and
$\phi(0)=\phi(\infty)=\infty$. This covering is, $\phi (X)= \frac {\sqrt{a}} {a-1} \frac {X^2-1} {X}$. The
branch points of $\phi$ are $q_i= \pm \frac {2i \sqrt{a}} {\sqrt{a-1}}$. From lemma \ref{lem1} the elliptic
subfields $E_1^\prime$ and $E_2^\prime$ have 2-torsion points $\{ 0, 1, -1,q_i\}$. The j-invariants of
$E_1^\prime$ and $E_2^\prime$ are
$$j^\prime= -16 \frac {(v-15)^3} {(v+1)^2}$$
Then $\Phi_2(j, j^\prime)=0$, so $E_1$ and $E_1^\prime$ are isogenous of degree 2. Thus, $\tau_1$ and
$\tau_3$ determine degree 2 elliptic subfields which are 2-isogenous.



\section{Jacobians of a genus two curves}

\def\u{\mathfrak u}
\def\v{\mathfrak v}



\section{The Kummer surface}
The Kummer surface is an algebraic variety which is quite useful in studying genus two curves. Using the Kummer surface we can take the Jacobian as  a double cover of the Kummer surface. Both the Kummer surface and the Jacobian, as noted above, can be given in terms of the theta functions and theta-nulls.

\subsection{Half Integer Theta Characteristics} For genus two curve, we have six odd theta
characteristics and ten even theta characteristics. The following are the sixteen theta characteristics where the first  ten are even and the last six are odd. For simplicity, we denote them by $\T_i(z)$ instead of $\T_i \ch{a} {b} (z, \t)$ where $i=1,\dots ,10$ for the even functions and $i=11, \dots, 16$ for the odd functions.
\[
\begin{split}
\T_1(z) &= \T_1 \chr {0}{0}{0}{0} (z , \t), \qquad \qquad \T_2(z) = \T_2 \chr {0}{0}{\frac{1}{2}} {\frac{1}{2}} (z , \t)\\
\T_3(z) &= \T_3 \chr {0}{0}{\frac{1}{2}}{0}(z , \t), \qquad \qquad
\T_4(z) = \T_4 \chr {0}{0}{0}{\frac{1}{2}} (z , \t)\\
\T_5(z) &= \T_5 \chr{\frac{1}{2}}{0} {0}{0}(z , \t), \qquad \qquad
\T_6(z) = \T_6 \chr {\frac{1}{2}}{0}{0}{\frac{1}{2}} (z , \t)\\
\T_7(z) &= \T_7 \chr{0}{\frac{1}{2}} {0}{0} (z , \t), \qquad \qquad
\T_8(z) = \T_8 \chr{\frac{1}{2}}{\frac{1}{2}} {0}{0} (z , \t)\\
\T_9(z) &= \T_9 \chr{0}{\frac{1}{2}} {\frac{1}{2}}{0}(z , \t), \qquad \qquad
\T_{10}(z) = \T_{10} \chr{\frac{1}{2}}{\frac{1}{2}} {\frac{1}{2}}{\frac{1}{2}} (z , \t)\\
\T_{11}(z) &= \T_{11} \chr{0}{\frac{1}{2}} {0}{\frac{1}{2}} (z , \t), \qquad \qquad
\T_{12}(z) = \T_{12} \chr{0}{\frac{1}{2}} {\frac{1}{2}}{\frac{1}{2}} (z , \t)\\
\T_{13}(z) &= \T_{13} \chr{\frac{1}{2}}{0} {\frac{1}{2}}{0} (z , \t), \qquad \qquad
\T_{14}(z) = \T_{14} \chr{\frac{1}{2}}{\frac{1}{2}} {\frac{1}{2}}{0} (z , \t)\\
\T_{15}(z) &= \T_{15} \chr{\frac{1}{2}}{0} {\frac{1}{2}}{\frac{1}{2}} (z , \t), \qquad \qquad
\T_{16}(z) = \T_{16} \chr{\frac{1}{2}}{\frac{1}{2}} {0}{\frac{1}{2}} (z , \t)\\
\end{split}
\]
%
\begin{rem}
All the possible half-integer characteristics except the zero characteristic can be obtained as the sum of not more
than 2 characteristics chosen from the following 5 characteristics:
\[ \left\{\chr {0}{0}{\frac{1}{2}} {\frac{1}{2}}, \chr {\frac{1}{2}}{0}{0}{\frac{1}{2}}, \chr{0}{\frac{1}{2}} {0}{0}, \chr{\frac{1}{2}}{0}
{\frac{1}{2}}{\frac{1}{2}},\chr{0}{\frac{1}{2}} {0}{\frac{1}{2}} \right\}.\]
The sum of all 5 characteristics in the set determines the zero characteristic,  where $\vartheta_{i}(z)$ are defined as
\[
\begin{split}
\vartheta_{1}(z) & =\T  \chr 0 0 0 0  (z,2\Omega), \quad \vartheta_{2}(z)=\T  \chr {\frac{1}{2}} {\frac{1}{2}} 0 0 \, (z, 2\Omega),\\
 \vartheta_{3}(z) & =\T \chr   0  {\frac{1}{2}} 0 0    \, (z,2\Omega),
\quad \vartheta_{4}(z)=\T \chr {\frac{1}{2}} 0 0 0 \,         (z,2\Omega)
\end{split}
\]
\end{rem}
see Shaska, Wijesiri \cite{ShW} for details.

\subsection{Inverting the Moduli Map} Let $\lambda_i,$ $i=1, \dots, n,$ be branch points of the genus
$g$ smooth curve $\X.$ Then the moduli map is a map from the configuration space $\Lambda$ of ordered $n$ distinct
points on $\P^1$ to the Siegel upper half space $\H_2$. In this section, we determine the branch points of genus 2
curves as functions of theta characteristics. The following lemma describes these relations using Thomae's formula. The
identities are known as Picard's formulas.
\begin{lem}[Picard] Let a genus 2 curve be given by
\begin{equation} \label{Rosen2}
Y^2=X(X-1)(X-\lambda)(X-\mu)(X-\nu).
\end{equation} Then, $\lambda, \mu, \nu$   can be written as follows:
%
\begin{equation}\label{Picard}
\l = \frac{\T_1^2\T_3^2}{\T_2^2\T_4^2}, \quad \mu = \frac{\T_3^2\T_8^2}{\T_4^2\T_{10}^2}, \quad \nu =
\frac{\T_1^2\T_8^2}{\T_2^2\T_{10}^2}.
\end{equation}
\end{lem}

\subsection{Kummer surface}

The Kummer surface is a variety obtained by grouping together two opposite points of the Jacobian of a genus 2 curve.  More precisely,
there is a map
\[\Psi: \Jac  (C) \rightarrow \K (C) \]
such that each point of $\K$ has two preimages which are opposite elements of $\Jac C$.  There are  16 exceptions that correspond to the 16 two-torsion points.  The Kummer surface does not naturally come with a group structure. However
the group law on the Jacobian endows a pseudo-group structure on the Kummer surface that is sufficient to define scalar multiplication.

\begin{definition}
Let $\Omega$ be a matrix in $\H_2$.  The Kummer surface associate to $\Omega$ is the locus of the images by the map $\varphi$ from $
\C^{2}$ to $\P^{3}(\mC)$ defined by:
\[\varphi: z\rightarrow \left(\T_{1}(2z),\T_{2}(2z), \T_{3}(2z), \T_{4}(2z)\right)\]
\end{definition}
Note that the Siegel upper half-space $\H_2$ is the set of symmetric $2\times2$ complex matrices with positive imaginary part which is
defined as
\[\H_2=\{F\in Mat(2,\C) \, | \, F=F^{T} \qquad  \textit{ and } \qquad \thinspace \Im  > 0\}\]
It can be proven that this map is well defined in the sense that the four $\T_{i}$ can not vanish simultaneously. The Theta functions
verify the following periodicity condition: for all $z\in \C^{2}$, $\forall b\in \{0,\frac{1}{2}\}^{2}$, and $\forall(m,n)\in \Z^{2}\times \Z^{2}$, we have
\[\T[0,b](z+\Omega m+n)=exp(-2\pi^{t}ibm-i\pi^{t}m\Omega m- 2i\pi^{t}mz)\cdot \T[0,b](z)\]
Therefore, two vectors that differ by an element of the lattice $\Z^{2}+\Omega\Z^{2}$ are mapped to the same point by $\varphi$. This map can be seen as a map from the Abelian variety $\C^{2}/(\Z^{2}+\Omega\Z^{2})$.
An additional result is that the Kummer surface of $\Omega$ is a projective variety of dimension 2 that we will denote by $\K(\Omega)$ or simply $\K$.  The group law on the Jacobian does not carry to a group law on $\K$. Indeed, since all $\T_{i}$ are even, $\varphi$ is even and maps two opposite element to the same point $P$.

We shall consider a Kummer surface $\K=\K_{a,b,c,d}$ parameterized by Theta constants:
\[
\begin{split}
& a=\T_{1}(0), \quad b=\T_{2}(0)\\
& c=\T_{3}(0), \quad d=\T_{4}(0)\\
\end{split}
\]
and
\[
\begin{split}
& A=\vartheta_{1}(0), \quad B=\vartheta_{2}(0)\\
& C=\vartheta_{3}(0), \quad D=\vartheta_{4}(0)\\
\end{split}
\]

Their squares are linked by simple linear relations that are obtained by putting $z=0$ in the equations above.
\[
\begin{split}
&4A^{2}=a^{2}+b^{2}+c^{2}+d^{2} \\
& 4B^{2}=a^{2}+b^{2}-c^{2}-d^{2} \\
& 4C^{2}=a^{2}-b^{2}+c^{2}-d^{2}\\
& 4D^{2}=a^{2}-b^{2}-c^{2}+d^{2}\\
\end{split}
\]
We write $(x,y,z,t)$ the projective coordinate of points on $\K$, that is:
\[x=\lambda\T_{1}(z), y=\lambda\T_{2}(z), z=\lambda\T_{3}(z), t=\lambda\T_{4}(z)\] for some $z\in \C^{2}$, and some $
\lambda\in\C^{*}$.\\
Then, the \textbf{Kummer surface} is given by the equation:
\begin{equation}
(x^{4}+y^{4}+z^{4}+t^{4})+2E xyzt-F(x^{2}t^{2}+y^{2}z^{2})-G(x^{2}z^{2}+y^{2}t^{2})-H(x^{2}y^{2}+z^{2}t^{2})=0
\end{equation}
where
\[
\begin{split}
E & =abcd\cdot \frac { A^{2}B^{2}C^{2}D^{2}} {(a^{2}d^{2}-b^{2}c^{2})(a^{2}c^{2}-b^{2}d^{2})(a^{2}b^{2}-c^{2}d^{2})} \\
F & = \frac {(a^{4}-b^{4}+d^{4})}{(a^{2}d^{2}-b^{2}c^{2})}\\
G & =  \frac {(a^{4}-b^{4}+c^{4}-d^{4})} {(a^{2}c^{2}-b^{2}d^{2}) } \\
 H & =   \frac {(a^{4}+b^{4}-c^{4}-d^{4})} { (a^{2}b^{2}-c^{2}d^{2})}\\
\end{split}
 \]

We have the following lemma.
\begin{lem}\label{possibleCurve}
Every genus two curve can be written in the form:
\[
y^2 = x \, (x-1) \, \left(x - \frac {\T_1^2 \T_3^2} {\T_2^2  \T_4^2}\right)\, \left(x^2 \, -   \frac{\T_2^2 \, \T_3^2 +
\T_1^2 \, \T_4^2} { \T_2^2 \, \T_4^2} \cdot    \a  \, x + \frac {\T_1^2 \T_3^2} {\T_2^2 \T_4^2} \, \a^2 \right),
\]
where $\a = \frac {\T_8^2} {\T_{10}^2}$ can be given in terms of $\, \, \T_1, \T_2, \T_3,$ and $\T_4$,
\[ \a^2 + \frac {\T_1^4 + \T_2^4 - \T_3^4 - \T_4^4}{\T_3^2 \T_4^2 - \T_1^2 \T_2^2  } \, \a + 1 =0. \]
Furthermore,  if $\alpha = {\pm} 1$ then $V_4 \embd Aut(\X)$.
\end{lem}

\begin{rem} {i)} From the above we have that $\T_8^4=\T_{10}^4$ implies that $V_4 \embd Aut(\X)$. Lemma \ref{lemma1}
determines a necessary and equivalent statement when $V_4 \embd Aut(\X)$.

{ii)}  The last part of Lemma~\ref{possibleCurve} shows that if $\T_8^4=\T_{10}^4$, then all coefficients of the genus 2 curve are
given as rational functions of the four fundamental theta functions. Such fundamental theta functions determine the
field of moduli of the given curve. Hence, the curve is defined over its field of moduli; see section 7 for details.
\end{rem}

\begin{cor}\label{cor1}
Let $\X$ be a genus 2 curve which has an elliptic involution. Then $\X$ is defined over its field of moduli.
\end{cor}

\subsection{Curves with automorphisms}

Since most of the computations on the Jacobians of genus two curves are performed using theta functions, and we are especially interested on genus two curves with split Jacobians it becomes desirable to describe the conditions that a genus two curve has extra automorphisms in terms of the theta functions. In other words to describe the loci with fixed automorphism group in terms of the theta functions.   The following was the main result of  \cite{ShW}.
\subsubsection{Describing the Locus of Genus Two Curves with Fixed Automorphism Group by Theta Constants}
The following lemma is proved in  \cite{ShW}
\begin{lem}\label{lemma1}
Let $\X$ be a genus 2 curve. Then  $Aut(\X)\iso V_4$ if and only if the theta functions of $\X$ satisfy

\begin{scriptsize}
\begin{equation}\label{V_4locus1}
\begin{split}
(\T_1^4-\T_2^4)(\T_3^4-\T_4^4)(\T_8^4-\T_{10}^4)
(-\T_1^2\T_3^2\T_8^2\T_2^2-\T_1^2\T_2^2\T_4^2\T_{10}^2+\T_1^4\T_3^2\T_{10}^2+ \T_3^2\T_2^4\T_{10}^2)\\
(\T_3^2\T_8^2\T_2^2\T_4^2-\T_2^2\T_4^4\T_{10}^2+\T_1^2\T_3^2\T_4^2\T_{10}^2-\T_3^4\T_2^2\T_{10}^2)
(-\T_8^4\T_3^2\T_2^2+\T_8^2\T_2^2\T_{10}^2\T_4^2\\
+\T_1^2\T_3^2\T_8^2\T_{10}^2
-\T_3^2\T_2^2\T_{10}^4)(-\T_1^2\T_8^4\T_4^2-\T_1^2\T_{10}^4\T_4^2+\T_8^2\T_2^2\T_{10}^2\T_4^2+\T_1^2\T_3^2\T_8^2\T_{10}^2)\\
(-\T_1^2\T_8^2\T_3^2\T_4^2+\T_1^2\T_{10}^2\T_4^4
+\T_1^2\T_3^4\T_{10}^2-\T_3^2\T_2^2\T_{10}^2\T_4^2)(-\T_1^2\T_8^2\T_2^2\T_4^2+\T_1^4\T_{10}^2\T_4^2\\
-\T_1^2\T_3^2\T_2^2\T_{10}^2+\T_2^4\T_4^2\T_{10}^2) (-\T_8^4\T_2^2\T_4^2
+\T_1^2\T_8^2\T_{10}^2\T_4^2-\T_2^2\T_{10}^4\T_4^2+\T_3^2\T_8^2\T_2^2\T_{10}^2)\\
(\T_1^4\T_8^2\T_4^2-\T_1^2\T_2^2\T_4^2\T_{10}^2-\T_1^2\T_3^2\T_8^2\T_2^2+\T_8^2\T_2^4\T_4^2)
 (\T_1^4\T_3^2\T_8^2-\T_1^2\T_8^2\T_2^2\T_4^2\\
-\T_1^2\T_3^2\T_2^2\T_{10}^2+\T_3^2\T_8^2\T_2^4)(\T_1^2\T_8^4\T_3^2-\T_1^2\T_8^2\T_{10}^2\T_4^2+\T_1^2\T_3^2\T_{10}^4
-\T_3^2\T_8^2\T_2^2\T_{10}^2)\\
(\T_1^2\T_8^2\T_4^4-\T_1^2\T_3^2\T_4^2\T_{10}^2+\T_1^2\T_3^4\T_8^2-\T_3^2\T_8^2\T_2^2\T_4^2)
  =0.
\end{split}
\end{equation}
\end{scriptsize}
\end{lem}
%

%
Our goal is to express each of the above loci in terms of the theta characteristics. We obtain the following result.
\begin{thm}\label{theorem1}
 Let $\X$ be a genus 2 curve. Then the following hold:

{i)}  $Aut(\X)\iso V_4$ if and only if the relations of theta functions given in Eq.~\eqref{V_4locus1} holds.

{ii)} $Aut(\X)\iso D_8$ if and only if the Eq. I in \cite{ShW}  is satisfied.

{iii)} $Aut(\X)\iso D_{12}$ if and only if the Eq. II and Eq. III in \cite{ShW} are satisfied.
\end{thm}

\begin{rem}
For e detailed account of the above see \cite{ShW} or \cite{Be3}.
\end{rem}

\subsection{Mapping a point of  $\K_{a,b,c,d}$ into the Jacobian of $\mC$}

Once we have an equation for the curve $\mC$ associated to $\K_{a,b,c,d}$, a natural question is to give an explicit function that maps the points of the Kummer surface to a class of divisors in the Jacobian of $\mC$, for instance in their Mumford representation.

Let $\K_{a,b,c,d}$ be a Kummer surface, and we assume that we have computed all the squares of Theta constants (under the condition that each of them will be non zero). Let $P=(x,y,z,t)$ be a point on $\K_{a,b,c,d}$. Then we can compute $\T_{i}(z)^{2}$ for $i\in [5,16]$,
corresponding to $(x,y,z,t)=(\T_{i}(z))_{i=1,2,3,4}.$

Then, let us define
\[ \u_{0}=\lambda \, \frac{\T_{8}^{2}\T_{14}(z)^{2}}{\T_{10}^{2}\T_{16}(z)^{2}}, \text{and} \thinspace u_{1}=(\lambda-1)\frac{\T_{5}^{2}\T_{13}(z)^{2}}{\T_{10}^{2}\T_{16}(z)^{2}}-u_{0}-1\]
We have noted $\T_{i}$ instead of $\T_{i}(0)$ for $i=1,2,3.....,16.$.

Since the Jacobian is of degree 2 of the Kummer surface, one should be able to decide which pair of opposite divisors is the real image of the point $P$, because for a given $\u$-polynomial, there are up to four $\v$-polynomials that yield a valid Mumford representation of a divisor in the Jacobian. These four $\v$-polynomials are grouped into two pairs of opposite divisors. Generically, giving the square of the
constant term of $v(x)=v_{1}x+v_{0}$ is enough to decide. Here is the formula for $v_{0}^{2}$ in terms of Theta functions:
\begin{small}
\[
\begin{split}
v_{0}^{2} & = -\frac {\T_{1}^{4}\T_{3}^{4}\T_{4}(z)^{2}} {(\T_{2}^{2}\T_{4}^{2}\T_{10}^{2}\T_{16}^{2}(z))^{3}}
\left( \T_{2}^{2}\T_{3}^{2} \T_{9}^{4}\T_{7}^{2}(z)\T_{12}^{2}(z) \T_{1}^{2}\T_{4}^{2}\T_{7}^{4}\T_{9}^{2}(z)\T_{11}^{2}(z) \right. \\
& \left.  +2\T_{1}^{2}\T_{2}^{2}\T_{3}^{2}(\T_{1}^{2}(z)\T_{3}^{2}(z)+\T_{2}^{2}(z)\T_{4}^{2}
(z)-2\T_{1}\T_{2}\T_{3}\T_{4}\T_{1}(z)\T_{2}(z)
\T_{3}(z)\T_{4}(z)(\T_{1}^{2}\T_{3}^{2}+\T_{2}^{2}\T_{4}^{2}) \right)
\end{split}
\]
\end{small}

Finally a formula for $v_{1}$ in terms of $v_{0},u_{0}$ and $u_{1}$ can be deduced from the fact that $u(x)$ should divide $v(x)^{2}-f(x)
$. If no theta constant vanishes, the map is undefined in the case where $\T_{16}(z)$ is zero. This corresponds to the case where the
image of the map is a divisor for which the $u$-polynomial in the Mumford representation is of degree less than 2 (i.e is a linear
function). Then, the formula in terms of theta functions is given by:
\[u_{0}=\frac{\lambda\T_{8}^{2}\T_{8}^{2}(z)}{(\lambda-1)\T_{5}^{2}\T_{12}^{2}(z)-\lambda\T_{8}^{2}\T_{14}^{2}(z)}.\]
Then, the Mumford representation of a divisor $D$ in the Jacobian is given $\langle x+u_{0},\pm\sqrt{f(-u_{0)}}\rangle$.

\section{Decomposable Jacobians}


Let $C$ be a genus 2 curve defined over an algebraically closed field $k$, of characteristic zero. Let $\psi: C
\to E$ be a degree $n$ maximal covering (i.e. does not factor through an isogeny) to an elliptic curve $E$ defined
over $k$. We say that $C$ has a \emph{degree n elliptic subcover}. Degree $n$ elliptic subcovers occur in pairs.
Let $(E, E')$ be such a pair. It is well known that there is an isogeny of degree $n^2$ between the Jacobian $J_C$
of $C$ and the product $E \times E'$. We say that $C$ has \textbf{(n,n)-split Jacobian}.

Curves of genus 2 with elliptic subcovers go back to Legendre and Jacobi. Legendre, in his \emph{Th\'eorie
des fonctions elliptiques}, gave the first example of a genus 2 curve with degree 2 elliptic subcovers. In a
review of Legendre's work, Jacobi (1832) gives a complete description for $n=2$. The case $n=3$ was studied
during the 19th century from Hermite, Goursat, Burkhardt, Brioschi, and Bolza. For a history and background
of the 19th century work see Krazer \cite[pg. 479]{Kr}. Cases when $n> 3$ are more difficult to handle. Recently,
Shaska dealt with cases $n=5, 7$ in \cite{deg5}.

The locus of $C$, denoted by $\L_n$, is an algebraic subvariety of the moduli space $\M_2$. The space $\L_2$ was studied
in Shaska/V\"olklein \cite{SV1}. The space $\L_n$ for $n=3, 5 $ was studied by Shaska in
\cite{deg3, deg5} were an algebraic description was given as sublocus of $\M_2$.

\subsection{Curves of genus 2 with split Jacobians}

Let $C$ and $E$ be curves of genus 2 and 1, respectively. Both are smooth, projective curves defined over $k$,
$char(k)=0$. Let $\psi: C \longrightarrow E$ be a covering of degree $n$. From the Riemann-Hurwitz formula,
$\sum_{P \in C}\, (e_{\psi}\,(P) -1)=2$ where $e_{\psi}(P)$ is the ramification index of points $P \in C$, under
$\psi$. Thus, we have two points of ramification index 2 or one point of ramification index 3. The two points of
ramification index 2 can be in the same fiber or in different fibers. Therefore, we have the following cases of
the covering $\psi$:\\

\textbf{Case I:} There are $P_1$, $P_2 \in C$, such that $e_{\psi}({P_1})=e_{\psi}({P_2})=2$, $\psi(P_1) \neq
\psi(P_2)$, and     $\forall  P \in C\setminus \{P_1,P_2\}$,  $e_{\psi}(P)=1$.

\textbf{Case II:} There are $P_1$, $P_2 \in C$, such that $e_{\psi}({P_1})=e_{\psi}({P_2})=2$, $\psi(P_1) =
\psi(P_2)$, and     $\forall  P \in C\setminus \{P_1,P_2\}$,  $e_{\psi}(P)=1$.

\textbf{Case III:} There is $P_1 \in C$ such that $e_{\psi}(P_1)=3$, and $ \forall P \in C \setminus \{P_1\}$,
$e_{\psi}(P)=1$.\\

\noindent In case I (resp. II, III) the cover $\psi$ has 2 (resp. 1) branch points in E.

Denote the hyperelliptic involution of $C$ by $w$. We choose $\mathcal O$ in E such that $w$ restricted to
$E$ is the hyperelliptic involution on $E$. We denote the restriction of $w$ on $E$ by $v$, $v(P)=-P$. Thus,
$\psi \circ w=v \circ \psi$. E[2] denotes the group of 2-torsion points of the elliptic curve E, which are
the points fixed by $v$. The proof of the following two lemmas is straightforward and will be omitted.

\begin{lem} \label{lem_1}
a) If $Q \in E$, then $\forall P \in \psi^{-1}(Q)$, $w(P) \in \psi^{-1}(-Q)$.

b) For all $P\in C$, $e_\psi(P)=e_\psi\,({w(P)})$.
\end{lem}

Let $W$ be the set of points in C fixed by $w$. Every curve of genus 2 is given, up to isomorphism, by a binary
sextic, so there are 6 points fixed by the hyperelliptic involution $w$, namely the Weierstrass points of $C$. The
following lemma determines the distribution of the Weierstrass points in fibers of 2-torsion points.

\begin{lem}\label{lem2} The following hold:
\begin{enumerate}
\item $\psi(W)\subset E[2]$

\item If $n$ is an odd number then

 i) $\psi(W)=E[2]$

 ii) If $ Q \in E[2]$ then \#$(\psi^{-1}(Q) \cap W)=1 \mod (2)$

\item If $n$ is an even number then for all $Q\in E[2]$, \#$(\psi^{-1}(Q) \cap W)=0 \mod (2)$
\end{enumerate}
\end{lem}

Let $\pi_C: C \lar \bP^1$ and $\pi_E:E \lar \bP^1$ be the natural degree 2 projections. The hyperelliptic
involution permutes the points in the fibers of $\pi_C$ and $\pi_E$. The ramified points of $\pi_C$, $\pi_E$
are respectively points in $W$ and $E[2]$ and their ramification index is 2. There is $\phi:\bP^1 \lar \bP^1$
such that the diagram commutes.
\begin{equation}
\begin{matrix}
C & \buildrel{\pi_C}\over\lar & \bP^1\\
\psi \downarrow &  & \downarrow \phi \\
E & \buildrel{\pi_E}\over\lar & \bP^1
\end{matrix}
\end{equation}
Next, we will determine the ramification of induced coverings $\phi:\bP^1 \lar \bP^1$. First we fix some
notation. For a given branch point we will denote the ramification of points in its fiber as follows. Any
point $P$ of ramification index $m$ is denoted by $(m)$. If there are $k$ such points then we write $(m)^k$.
We omit writing symbols for unramified points, in other words $(1)^k$ will not be written. Ramification data
between two branch points will be separated by commas. We denote by $\pi_E (E[2])=\{q_1, \dots , q_4\}$ and
$\pi_C(W)=\{w_1, \dots ,w_6\}$.

\subsection{Maximal coverings $\psi:C \lar E$.}

Let $\psi_1:C \lar E_1$ be a covering of degree $n$ from a curve of genus 2 to an elliptic curve. The
covering $\psi_1:C \lar E_1$ is called a \textbf{maximal covering} if it does not factor through a nontrivial
isogeny. A map of algebraic curves $f: X \to Y$ induces maps between their Jacobians $f^*: J_Y \to J_X$ and
$f_*: J_X \to J_Y$. When $f$ is maximal then $f^*$ is injective and $ker (f_*)$ is connected, see \cite{Sh7}
for details.

Let $\psi_1:C \lar E_1$ be a covering as above which is maximal. Then ${\psi^*}_1: E_1 \to J_C$ is injective
and the kernel of $\psi_{1,*}: J_C \to E_1$ is an elliptic curve which we denote by $E_2$. For a fixed Weierstrass point $P \in C$, we can embed $C$ to its Jacobian via
\begin{equation}
\begin{split}
i_P: C & \lar J_C \\
 x & \to [(x)-(P)]
\end{split}
\end{equation}
Let $g: E_2 \to J_C$ be the natural embedding of $E_2$ in $J_C$, then there exists $g_*: J_C \to E_2$. Define
$\psi_2=g_*\circ i_P: C \to E_2$. So we have the following exact sequence
$$ 0 \to E_2 \buildrel{g}\over\lar J_C \buildrel{\psi_{1,*}}\over\lar E_1 \to 0 $$
The dual sequence is also exact
$$ 0 \to E_1 \buildrel{\psi_1^*}\over\lar J_C \buildrel{g_*}\over\lar E_2 \to 0 $$
If $deg (\psi_1)$ is an odd number then the maximal covering $\psi_2: C \to E_2$ is unique. If the cover $\psi_1:C \lar E_1$ is given, and therefore $\phi_1$,
we want to determine $\psi_2:C \lar E_2$ and $\phi_2$. The study of the relation between the ramification
structures of $\phi_1$ and $\phi_2$ provides information in this direction. The following lemma (see
 answers this question for the set of Weierstrass points $W=\{P_1, \dots , P_6\}$ of C
when the degree of the cover is odd.

\begin{lem} Let $\psi_1:C \lar E_1$, be maximal of  degree $n$.
Then, the map $\psi_2: C\to E_2$ is a maximal covering of degree $n$. Moreover,
\begin{enumerate}
\item [i) ] if  $n$ is  odd and ${\cO}_i\in E_i[2]$, $i=1, 2$  are  the places such that $\#
(\psi_i^{-1}({\cO }_i)\cap W) = 3$, then $\psi_1^{-1}({\cO }_1)\cap W$ and $\psi_2^{-1}({\cO }_2)\cap W$ form
a disjoint union of $W$.
\item [ii)] if $n$ is even and $Q\in E[2]$, then $\# \left( \psi^{-1}(Q)\right) = 0$ or 2.
\end{enumerate}
\end{lem}
The above lemma says that if $\psi$ is maximal of even degree then the corresponding induced covering can
have only type \textbf{I} ramification.

\subsection{The locus of genus two curves with $(n, n)$ split Jacobians}


Two covers $f:X\to \bP^1$ and $f':X'\to \bP^1$ are called \textbf{weakly equivalent} if there is a
homeomorphism $h:X\to X'$ and an analytic automorphism $g$ of $\bP^1$ (i.e., a Moebius transformation) such
that  $g\circ f=f'\circ h$. The covers $f$ and $f^\prime$ are called \textbf{equivalent } if the above holds
with $g=1$.

Consider a cover $f:X \to \bP^1$ of degree $n$, with branch points $p_1,...,p_r\in \bP^1$. Pick $p\in \bP^1
\setminus\{p_1,...,p_r\}$, and choose loops $\gamma_i$ around $p_i$ such that $\gamma_1,...,\gamma_r$ is a
standard generating system of the fundamental group $\Gamma:=\pi_1( \bP^1 \setminus\{p_1,...,p_r\},p)$, in
particular, we have $\gamma_1 \cdots \gamma_r=1$. Such a system $\gamma_1,...,\gamma_r$ is called a homotopy basis
of $\bP^1 \setminus\{p_1,...,p_r\}$. The group $\Gamma$ acts on the fiber $f^{-1}(p)$ by path lifting, inducing a
transitive subgroup $G$ of the symmetric group $S_n$ (determined by $f$ up to conjugacy in $S_n$). It is called
the \textbf{monodromy group} of $f$. The images of $\gamma_1,...,\gamma_r$ in $S_n$ form a tuple of permutations
$\s=(\s_1,...,\s_r)$ called a tuple of \textbf{branch cycles} of $f$.

We say a cover $f:X\to\bP^1$ of degree $n$ is of type $\s$ if it has $\s$ as tuple of branch cycles relative to
some homotopy basis of $\bP^1$ minus the branch points of $f$. Let $\Hs$ be the set of weak equivalence classes of
covers of type $\s$.  The \textbf{Hurwitz space} $\Hs$ carries a natural structure of an quasiprojective variety.

We have $\Hs=\H_\tau$ if and only if the tuples $\s$, $\tau$ are in the same \textbf{braid orbit} $\mathcal
O_\tau = \mathcal O_\sigma$. In the case of the covers $\phi : \bP^1 \to \bP^1$ from above, the corresponding
braid orbit consists of all tuples in $S_n$ whose cycle type matches the ramification structure of $\phi$.

\subsubsection{Humbert surfaces}
Let $\A_2$ denote the moduli space of principally polarized Abelian surfaces. It is well known that $\A_2$ is
the quotient of the Siegel upper half space $\H_2$ of symmetric complex $2 \times 2$ matrices with
positive definite  imaginary part by the action of the symplectic group $Sp_4 (\Z)$; see \cite[p. 211]{G}.

Let $\D$ be a fixed positive integer and  $N_\D$ be the set of matrices
$$\tau =
\begin{pmatrix}z_1 & z_2\\
z_2 & z_3
\end{pmatrix}
\in \mathfrak H_2$$ such that there exist nonzero integers $a, b, c, d, e $ with the following properties:
\begin{equation}\label{humb}
\begin{split}
& a z_1 + bz_2 + c z_3 + d( z_2^2 - z_1 z_3) + e = 0\\
& \D= b^2 - 4ac - 4de\\
\end{split}
\end{equation}

The \emph{Humbert surface}  $\H_\D$  of discriminant $\D$   is called the image of $N_\D$ under the
canonical map
$$H_2 \to \A_2:= Sp_4( \Z)\setminus{ H_2},$$
see \cite{Hu, BW, Mu} for details.  It is known that $\H_\D \neq \emptyset$ if and only if $\D
> 0$ and $\Delta \equiv 0 \textit { or } 1 \mod 4$. Humbert (1900) studied the zero loci in
Eq.~\eqref{humb} and discovered certain relations between points in these spaces and certain plane
configurations of six lines; see \cite{Hu} for more details.

For a genus 2 curve $C$ defined over $\bC$, $[C]$ belongs to $\L_n$ if and only if the isomorphism class $[J_C]
\in \A_2$ of its (principally polarized) Jacobian $J_C$ belongs to the Humbert surface $\H_{n^2}$, viewed as a
subset of the moduli space $\A_2$ of principally polarized Abelian surfaces; see \cite[Theorem 1, p. 125]{Mu}  for
the proof of this statement. In \cite{Mu} is shown that there is a one to one correspondence between the points in
$\L_n$ and points in $\H_{n^2}$. Thus, we have the map:
\begin{equation}
\begin{split}
&  \H_\s \, \,  \longrightarrow \, \, \L_n  \, \, \longrightarrow  \, \, \H_{n^2}\\
([f], (p_1, & \dots , p_r)  \to [\X] \to [J_\X]\\
\end{split}
\end{equation}
In particular, every point in $\H_{n^2}$ can be represented by an element of $\mathfrak H_2$ of the form
$$\tau =
\begin{pmatrix}z_1 & \frac 1 n \\
\frac 1 n & z_2
\end{pmatrix}, \qquad z_1, \, z_2 \in \mathfrak H.
$$
There have been many attempts to explicitly describe these Humbert surfaces. For some small discriminant this
has been done in \cite{SV1}, \cite{deg3}, \cite{deg5}. Geometric characterizations of such
spaces for $\D= 4, 8, 9$, and 12 were given by Humbert (1900) in \cite{Hu} and for $\D= 13, 16, 17, 20$, 21
by Birkenhake/Wilhelm.

\subsection{Genus 2 curves with degree 3 elliptic subcovers}

This case was studied in detail in\cite{deg3}. The main theorem was:

\begin{thm}\label{main_thm_deg3}
Let $K$ be a genus 2 field and $e_3(K)$ the number of $Aut(K/k)$-classes
 of elliptic subfields of $K$ of degree 3.  Then;

i)  $e_3(K) =0, 1, 2$, or  $4$

ii)    $e_3(K) \geq 1$ if and only if
 the classical invariants of $K$ satisfy  the irreducible
equation   $F(J_2, J_4, J_6, J_{10})=0$ displayed  in  \cite[Appendix A]{deg3}.
\end{thm}

There are exactly two genus 2 curves (up to isomorphism) with $e_3(K)=4$. The case $e_3(K)=1$ (resp., 2) occurs
for a 1-dimensional (resp., 2-dimensional)  family of genus 2 curves, see \cite{deg3}.

\begin{figure}[htbp]
\begin{center}
\includegraphics[width=8cm]{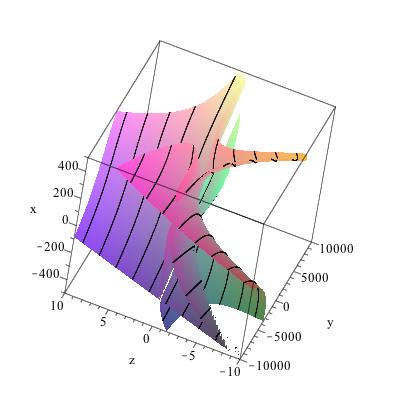}
\caption{Shaska's surface as graphed in \cite{Be1}}
\label{default}
\end{center}
\end{figure}

A geometrical interpretation of the Shaska's surface (the space $\L_3$) and its singular locus can be found in \cite{Be1}.

\begin{lem}\label{lemma3_1}
Let $K$ be a genus 2 field and $E$ an elliptic subfield of degree 3.

i) Then $K=k(X,Y)$ such that
\begin{equation}\label{n_3_form1}
Y^2= (4X^3+b^2X^2+2bX+1) (X^3+aX^2+b X+1)
\end{equation}
for $a, b \in k$ such that
\begin{equation}\label{n_3_form2}
\begin{split}
(4a^3+27-18ab-a^2b^2+4b^3)(b^3-27) \neq 0
\end{split}
\end{equation}
The roots of the first (resp. second) cubic correspond to $W^{(1)}(K,E)$, (resp. $W^{(2)}(K,E)$) in the
coordinates $X,Y$, (see theorem \ref{thm1}).

ii) $E=k(U,V)$ where $$U= \frac {X^2} {X^3+aX^2+b X+1}$$ and
\begin{equation}\label{eq_E1}
V^2= U^3 +2 \frac {ab^2-6a^2+9b} {R} U^2 + \frac {12a-b^2 } {R} U - \frac 4 {R}
\end{equation}
where $R=4a^3+27-18ab-a^2b^2+4b^3 \neq 0$.

iii) Define
$$ u:=ab, \quad v:=b^3 $$
Let $K^\prime$ be a genus 2 field and $E^\prime \subset K^\prime$ a degree 3 elliptic subfield. Let
$a^\prime, b^\prime$ be the associated parameters as above and $u^\prime:=a^\prime b^\prime$,
$v=(b^\prime)^3$. Then, there is a $k$-isomorphism $K \to K^\prime$ mapping $E\to E^\prime$ if and only if
exists a third root of unity $\e \in k$ with $a^\prime=\e a$ and $b^\prime=\e^2 b$. If $b\neq 0$ then such
$\e$ exists if and only if $v=v^\prime$ and $u=u^\prime$.

iv) The classical invariants of $K$ satisfy equation \cite[Appendix A]{deg3}.
\end{lem}

Let
\begin{equation}
\begin{split}
F(X) & :=X^3+a X^2+b X+1\\
G(X) & :=4X^3+b^2X^2+2bX+1
\end{split}
\end{equation}

Denote by $R=4a^3+27-18ab-a^2b^2+4b^3$ the resultant of $F$ and $G$. Then we have the following lemma.

\begin{lem}\label{lemma3_2} Let $a,b \in k$ satisfy equation \eqref{n_3_form2}.
Then equation \eqref{n_3_form1} defines a genus 2 field $K=k(X,Y)$. It has elliptic subfields of degree 3,
$E_i=k(U_i, V_i)$, $i=1,2$, where $U_i$, and $V_i$ are as follows:
$$
U_1 = \frac {X^2} {F(X)}, \quad V_1= Y\, \frac {X^3-bX-2} {F(X)^2}
$$

\begin{small}
\begin{equation}
\begin{split}
U_2 & = \left\{ \aligned
  \frac {(X-s)^2 (X-t)} {G(X) } &\quad if \quad b(b^3-4ba+9) \neq 0 \\
  \frac { (3X- a)} {3(4X^3+1)} & \quad if \quad b= 0\\
  \frac { (bX+3)^2} {b^2G(X)}  &\quad if \quad (b^3-4ba+9) = 0\\
\endaligned
\right.
\end{split}
\end{equation}
\end{small}

where
$$ s=- \frac 3 b, \quad t=\frac {3a-b^2} {b^3-4ab+9}$$

\begin{scriptsize}
\begin{equation}
\begin{split}
V_2 & = \left\{ \aligned
 \frac {\sqrt {27-b^3} Y} {G(X)^2}
((4ab-8-b^3)X^3 -(b^2-4ab)X^2 +bX+1) & \quad if \quad b(b^3-4ba+9) \neq 0\\
 Y \frac {8X^3-4aX^2-1} {(4X^3+1)^2}   & \quad if \quad b= 0\\
\frac 8 b \sqrt{b} \frac Y {G(X)} (bX^3+9X^2+b^2X+b )    & \quad if \quad (b^3-4ba+9) = 0\\
\endaligned
\right.
\end{split}
\end{equation}
\end{scriptsize}
\end{lem}

\subsection{Elliptic subcovers}

We express the j-invariants $j_i$ of the elliptic subfields $E_i$ of $K$, from Lemma~\ref{lemma3_2}, in terms
of $u$ and $v$ as follows:

\begin{small}
\begin{equation}\label{j_1}
\begin{split}
j_1 &= 16v \frac {(vu^2+216u^2-126vu-972u+12v^2+405v)^3}
{ (v-27)^3(4v^2+27v+4u^3-18vu-vu^2)^2}\\
j_2 &= -256 \frac {(u^2-3v)^3}{v(4v^2+27v+4u^3-18vu-vu^2)} \\
\end{split}
\end{equation}
\end{small}
where $v\neq 0, 27$.
\begin{remark}
The automorphism $\nu \in Gal_{k(u,v)/k(\r_1, \r_2)}$ permutes the elliptic subfields. One can easily check
that:
$$\nu(j_1)=j_2, \quad \nu(j_2)=j_1$$
\end{remark}

\begin{lem}
The j-invariants of the elliptic subfields satisfy the following quadratic equations over $k(r_1, r_2)$;
\begin{equation}\label{eq_j_new}
j^2- T \, j+ N =0,
\end{equation}
where $T, N$ are given in \cite{deg3}.
\end{lem}

\subsubsection{Isomorphic Elliptic Subfields}
Suppose that $E_1\iso E_2$. Then, $j_1=j_2$ implies that
\begin{small}
\begin{equation}
\begin{split}
8v^3+27v^2-54uv^2-u^2v^2+108u^2v+4u^3v-108u^3=0\\
\end{split}
\end{equation}
\end{small}
or
\begin{scriptsize}
\begin{equation}
\begin{split}\label{n_3_iso2}
& 324v^4u^2-5832v^4u+37908v^4-314928v^3u-81v^3u^4+255879v^3+30618v^3u^2\\
& -864v^3u^3-6377292uv^2 +8503056v^2-324u^5v^2+2125764u^2v^2-215784u^3v^2\\
& +14580u^4v^2+16u^6v^2+78732u^3v+8748u^5v -864u^6v-157464u^4v+11664u^6 =0\\
\end{split}
\end{equation}
\end{scriptsize}

The former equation is the condition that $det ( Jac(\th))=0$. The
expressions of $i_1, i_2, i_3$  we can express $u$ as a rational function in $i_1, i_2$, and $v$. This is
displayed in \cite[Appendix B]{deg3}. Also, $[k(v):k(i_1)]=8$ and $[k(v):k(i_2)]=12$. Eliminating $v$ we get a
curve in $i_1$ and $i_2$ which has degree 8 and 12 respectively. Thus, $k(u,v)=k(i_1,i_2)$. Hence, $e_3(K) = 1$
for any $K$ such that the associated $u$ and $v$ satisfy the equation; see \cite{deg3} for details.

\subsubsection{The Degenerate Case}
We assume now that one of the extensions $K/E_i$ from Lemma~\ref{lemma3_2} is degenerate, i.e. has only one
branch point. The following lemma determines a relation between $j_1$ and $j_2$.

\begin{lem} Suppose that $K/E_2$ has only one branch point. Then,
$$729 j_1 j_2 -(j_2-432)^3=0$$
\end{lem}

For details of the proof see Shaska \cite{deg3}.
Making the substitution $T=-27j_1$ we get $$j_1=F_2(T)= \frac {(T+16)^3} T$$ where $F_2(T)$ is the Fricke
polynomial of level 2.

If both $K/E_1$ and $K/E_2$ are degenerate then
\begin{equation}
\begin{split}
\left\{ \aligned
729j_1j_2-(j_1-432)^3=0\\
729j_1j_2-(j_2-432)^3=0\\
\endaligned
\right.
\end{split}
\end{equation}
There are 7 solutions to the above system. Three of which give isomorphic elliptic curves
$$j_1=j_2=1728, \quad j_1=j_2=\frac 1 2 (297 \pm 81 \sqrt{-15})$$
The other 4 solutions are given by:
\begin{equation}
\begin{split}
\left\{ \aligned
729j_1j_2-(j_1-432)^3=0\\
j_1^2+j_2^2-1296(j_1+j_2)+j_1 j_2 +559872=0\\
\endaligned
\right.
\end{split}
\end{equation}

\subsection{Further remarks}

If $e_3(C)\geq 1$ then the automorphism group of $C$ is one of the following: $\bZ_2, V_4$, $ D_4$, or $D_6$.
Moreover; there are exactly 6 curves $C\in \L_3$ with automorphism group $D_4$ and six curves $C\in \L_3$ with
automorphism group $D_6$. They are listed in \cite{Sh6} where rational points of such curves are found.

Genus 2 curves with  degree 5 elliptic subcovers are studied in \cite{deg5} where a description of the space $\L_5$
is given and all its degenerate loci. The case of degree 7 is the first case when all possible degenerate loci
occur.

We have organized the results of this paper in a Maple package which determines if a genus 2 curve has degree
$n=2,3$ elliptic subcovers. Further, all its elliptic subcovers are determined explicitly. We intend to implement
the results for $n=5$ and the degenerate cases for $n=7$.

\section{Field of moduli versus the field of definition}

\def\Hm{\mathcal H}

Let $\X$ be a curve defined over $k$. A field $F \subset k$ is called a {\it field of definition} of $\X$ if there
exists $\X'$ defined over $F$ such that $\X \iso \X'$. The {\bf field of moduli} of $\X$ is a subfield $F \subset
k$ such that for every automorphism $\sigma$ of $k$, $\X$ is isomorphic to $\X^\sigma$ if and only if $ \sigma_F =
id$.

The field of moduli is not necessary a field of definition. To determine the points $\p \in \M_g$ where the field
of moduli is not a field of definition is a classical problem in algebraic geometry and has been the focus of many
authors, Weil, Shimura, Belyi, Coombes-Harbater, Fried, D\'ebes, Wolfart among others.

Weil (1954) showed that for every algebraic curve with trivial automorphism group, the field of moduli is a field
of definition. Shimura (1972) gave the first example of a family of curves such that the field of moduli is not a
field of definition. Shimura's family were a family of hyperelliptic curves.  Further he adds:\\

 {\it `` ... the
above results combined together seem to indicate a rather complicated nature of the problem, which almost defies
conjecture. A new viewpoint is certainly necessary to understand the whole situation''}\\

We call a point $\p \in \Hm_g$ a {\it moduli point}. The field of moduli of $\p$ is denoted by $F_\p$. If there is
a curve $\X_g$ defined over $F_\p$ such that $\p=[\X_g]$, then we call such a curve a {\it rational model over the
field of moduli}. Consider the following problem:

\smallskip

{\it Let the moduli point $\p \in \Hm_g$ be given. Find necessary and sufficient conditions that the field of
moduli $F_\p$ is a field of definition. If $\p$ has a rational model $\X_g$ over its field of moduli, then
determine explicitly the equation of $\X_g$.}

\smallskip

In 1993, Mestre solved the above problem for genus two curves with automorphism group $\Z_2$. In Corr.~\ref{cor1} is proved  that for points $\p \in \M_2$ such that $| Aut (\p)| > 2$ the field
of moduli is a field of definition

The proof of the above facts is constructive. In other words, a rational model is given. In the case when the field of moduli is not a field of definition a rational model is given over the minimal field of definition.

More generally one can consider the following problem for genus $g > 2$ hyperelliptic curves.

\section{Factoring large numbers with genus 2 curves}

In \cite{Co} an algorithm is suggested for factoring large numbers using genus two curves. Such algorithm chooses genus two curves with (2, 2)-split Jacobians. 

\subsection{Algorithm}

HECM begin by computing $k=\prod_{ \pi\leq B_{1} } \pi^{ \left[  \log\left(B_{1}\right) / \log\left(\pi\right)  \right] }.$
We hope to encounter the zero of one of the underlying elliptic curves so it is important to have explicit morphisms between the Kummer surface and the two underlying elliptic curves. If the elliptic curves are in the Weierstrass form, we only need the coordinates $\left(x::z\right)$ to test if the point is zero: just compute $\gcd\left(z,n\right)$. The morphisms between the Kummer surface and the underlying elliptic curves are rational over $\Q$.

\begin{algorithm}[t]
  \caption{HECM (stage 1)}
  \label{HECM}
\begin{algorithmic}[1]

	\REQUIRE the number $n$ to be factor. The smoothness bound $B_{1}$.
	\ENSURE a factor $p$ of $n$.
	
	\STATE Compute $ k=lcm\left(1,2,\ldots,B_{1}\right) $.
	\STATE Choose a random decomposable curve $\mC$ over $\Z/n\Z$ and a point $P$ on its Kummer surface.
	\STATE Compute $Q=[k]P$.
	\STATE Map $Q$ to the two underlying elliptic curves $\mE_{i}$.
	\STATE Hope that $Q=\mathcal{O}_{\left(\mE_{i}\right)} \mod p$ for one $\mE_{i}$ (test whether $\gcd\left(z,n\right)\neq 1$).
	\STATE Else go to $2$.

\end{algorithmic}
\end{algorithm}

The global morphism in HECM goes from the Kummer surface to $\mE_{1}\times\mE_{2}$.

Let $ Q=(x,y,z,t)$ be a point on the Kummer surface $\K_{a,b,c,d}$ corresponding to the hyperellipitc curve $\mC$ of equation $y^{2}=f(x) $. We want to map this point in the Jacobian of the hyperelliptic curve. As in the above section we find the Mumford coordinates $(u,v)$ of two opposite divisors $\Psi^{-1}(Q)={\pm D}$.

Now, let $P$ be a point on the $(2,2)$-decomposable hyperelliptic curve $\mC$ given by:\[\mC: \chi y^{2}=x(x-1)(x-\lambda)(x-\mu)(x-\nu).
\]
We want to map $P$ to the two elliptic curves $\mE_{1},\mE_{2}$.  Let $\mC^{\prime}$ be the curve given by  \[\mC^{\prime}: \kappa y^2=
\left(x^{2}-1\right)\left(x^{2}-x_2^2\right)\left(x^{2}-x_3^2\right)  \]
with \[ q= \pm \sqrt{\mu\left(\mu-\nu\right)},\quad x_2= \frac{\mu+ q}{\mu- q}, \ x_{3}=\frac{1-\mu - q}{1-\mu + q}, \ \chi= -q\kappa \mu
\left(\mu-1\right).\] The curves $\mC$ and $\mC^{\prime}$ are isomorphism by the change of the coordinates. Than the curve $\mC^{\prime}$
maps to the elliptic curve
\[\mE: y^2=\left(x-1\right)\left(x-x_2^2\right)\left(x-x_3^2\right)  \]
by the morphism $(x,y)\rightarrow (x^{2},y)$. Let $f$ be the map from $\mC$ to $\mE_{1}\times \mE_{2}$. The push forward $f_{*}$ of $f$ is defined:

\begin{equation*}
f_{*}:\left\{
\begin{array}{rl}
\J\longrightarrow Jac(\mE_{1})\times Jac(\mE_{2})\\
D=\sum_{i=1}^{r}P_{i}-rP_{\infty}\longrightarrow \sum_{i=1}^{r}f(P_{i})-rf(P_{\infty})
\end{array}\right.
\end{equation*}
where $f(P\infty)=(\O_{\mE_{1}},\O_{\mE_{2}})$ are the zeros of the two elliptic curves.  Since the divisors in the Jacobian of the elliptic curve are not reduced, and it is isomorphic to the set of points on the curve, then the function $f_{*}$ can be rewritten as:

\begin{equation*}
f_{*}:\left\{
\begin{array}{rl}
\J\longrightarrow \mE_{1}\times\mE_{2}\\
D=\sum_{i=1}^{r}P_{i}-rP_{\infty}\longrightarrow \sum_{i=1}^{r}f(P_{i})
\end{array}\right.
\end{equation*}

We now turn to what is called stage $2$. The initial point for the arithmetic is $Q=[k]P$, thus we don't have the benefit of a  "good" initial point. Moreover, stage $2$ needs the $x$-coordinate of many points on the elliptic curve in Weierstrass form \cite{ZiDo06},
therefore if we use hyperelliptic curves we need to apply morphisms a lot and, in this case, their cost would not be negligible.  For all these reasons, it seems that hyperelliptic curves should not be used for stage~$2$. Instead, we can apply the ECM stage $2$ to the two
underlying elliptic curves.\\

For stage 2, ECM method needs a point on the curve $y^{2}=x^{3}+Ax+B$, but we have a point $(x_{1}:z_{1})$ on
$\kappa y^{2}z=x^{3}+a_{2}x^{2}z+a_{4}xz^{2}+a_{6}z^{3}$.  First translate the point by putting $x\rightarrow x-a_{2}z/3$ and dividing the
x-coordinate by $z$ to get a point $(x_{1}^{\prime}, \Box)$ on the curve $\kappa y^{2}=f(x)=x^{3}+a_{4}^{'}+a_{6}^{'}$. Then the points $
(x_{1}^{\prime},\pm 1)$ are on the curve
\[T y^{2}=f(x)=x^{3}+a_{4}^{\prime}x+a_{6}^{'}\] with $D=f(x_{1}^{\prime}.$
By the change of variable
\[(x,y)\mapsto (\frac{x}{T},\frac{y}{T}) \]
we get a point on the curve $y^{2}=x^{3}+Ax+B$ with $A=a_{4}^{\prime}/T$ and $B=a_{6}^{\prime}/T^{3}.$

\section{A computational package for genus two curves}\label{genus2}

Genus 2 curves are the most used of all hyperelliptic curves due to their application in cryptography and
also best understood. The moduli space $\M_2$ of genus 2 curves is a 3-dimensional variety. To understand
how to describe the moduli points of this space we need to define the invariants of binary sextics. For
details on such invariants and on the genus 2 curves in general the reader can check \cite{Ig}, \cite{Sh3},
\cite{KSV}.

\begin{equation}
i_1:=144 \frac {J_4} {J_2^2}, \quad i_2:=- 1728 \frac {J_2J_4-3J_6} {J_2^3}, \quad i_3 :=486 \frac {J_{10}}
{J_2^5},
\end{equation}
for $J_2 \neq 0$. In the case $J_2=0$ we define
\begin{equation}\label{eq_a_1_a_2}
\a_1:= \frac {J_4 \cdot J_6} {J_{10}}, \quad \a_2:=\frac {J_6 \cdot J_{10}} {J_4^4}
\end{equation}
to determine genus two fields with $J_2=0$, $J_4\neq 0$, and $J_{6}\neq 0$ up to isomorphism.

For a given genus 2 curve $C$ the corresponding \textbf{moduli point} $\p = [C]$ is defined as
\[
\begin{split}
\p = \left\{
\aligned
& (i_1, i_2, i_3) \; \; if \; J_2 \neq 0\\
& (\a_1, \a_2) \; \; if \; J_2=0, J_4\neq 0, J_{6}\neq 0 \\
& \frac{J_{6}^5}{J_{10}^3} \;\; if \; J_2=0, J_4=0, J_6\neq 0  \\
& \frac{J_{4}^5}{J_{10}^2} \;\; if \; J_2=0, J_6=0, J_4\neq 0\\
\endaligned
\right.
\end{split}
\]
Notice that the definition of $\a_1, \a_2$ can be totally avoided if one uses absolute invariants with
$J_{10}$ in the denominator. However, the degree of such invariants is higher and therefore they are not
effective computationally.

We have written a Maple package which finds most of the common properties and invariants of genus two curves. While this is still work in progress, we will describe briefly some of the functions of this package. The functions in this package are:\\

\noindent \verb"J_2, J_4, J_6, J_10, J_48, L_3_d,  a_1, a_2, i_1, i_2, i_3,"

\noindent  \verb"theta_1, theta_2, theta_3, theta_4, AutGroup, CurvDeg3EllSub_J2,"

\noindent   \verb"CurveDeg3EllSub, Ell_Sub, LocusCurves,Aut_D4, LocusCurvesAut_D4_J2,"

\noindent    \verb"LocusCurvesAut_D6, LocusCurvesAut_V4, Rational_Model, Kummer."

\medskip

Next, we will give some examples on how some of these functions work.

\subsection{Automorphism groups}\label{sec:autgroup}
A list of groups that can occur as automorphism groups of hyperelliptic curves is given in \cite{Sh5} among
many other references. The function in the package that computes the automorphism group is given by
$AutGroup ()$. The output is the automorphism group. Since there is always confusion on the terminology when
describing certain groups we also display the GAP identity of the group from the \verb"SmallGroupLibrary".

For a fixed group $G$ one can compute the locus of genus $g$ hyperelliptic curves with automorphism group
$G$. For genus 2 this loci is well described as subvarieties of $\mathcal M_2$.

\begin{example} Let $y^2=f(x)$ be a genus 2 curve where $f:=x^5+2x^3-x$. Then the function
 \verb"AutGroup"(f,x) displays:\\

\noindent $>$ \verb"AutGroup"(f,x);
\begin{center} $[D_4,(8,3)]$ \end{center}

\end{example}

\begin{example} Let $y^2=f(x)$ be a genus 2 curve where  $f:=x^6+2x^3-x$. Then the function
\verb"AutGroup"(f,x) displays:\\

\noindent $>$ \verb"AutGroup"(f,x);
\begin{center}    $[V_4,(4,2)]$ \end{center}
\end{example}

\noindent We also have implemented the functions: \verb"LocusCurvesAut_V_4()",

\verb"LocusCurvesAut_D_4()", \verb"LocusCurvesAut_D4_J2()", \verb"LocusCurvesAut_D_6()",

\noindent which gives equations for the locus of curves with automorphism group $D_4$ or $D_6$.

\subsection{Genus 2 curves with split Jacobians}
A genus 2 curve which has a degree $n$ maximal map to an elliptic curve is said to have $(n, n)$-\emph{split
Jacobian}; see \cite{Sh6} for details. Genus 2 curves with split Jacobian are interesting in number theory,
cryptography, and coding theory. We implement an algorithm which checks if a curve has $(3, 3)$, and
$(5,5)$-split Jacobian. The case of $(2, 2)$-split Jacobian corresponds to genus 2 curves with extra
involutions and therefore can be determined by the function \verb"LocusCurvesAut_V_4"().

The function which determines if a genus 2 curve has $(3, 3)$-split Jacobian is \verb"CurvDeg3EllSub"() if
the curve has $J_2 \neq 0$ and \verb"CurvDeg3EllSub_J_2 ()" otherwise; see \cite{Be5}. The input of
\verb"CurvDeg3EllSub"() is the triple $(i_1, i_2, i_3)$ or the pair $(\a_1, \a_2)$ for
\verb"CurvDeg3EllSub_J_2 ()". If the output is 0, in both cases, this means that the corresponding curve to
this moduli point has $(3, 3)$-split Jacobian. Below we illustrate with examples in each case.

\begin{example} Let $y^2=f(x)$ be a genus 2 curve where $f:=4 x^6+9 x^5+8 x^4+10 x^3+5 x^2+3 x+1$. Then,\\

 \noindent \verb"> i_1:=i_1(f,x); i_2:=i_2(f,x); i_3:=i_3(f,x);"
 \[i_1:= \frac{78741}{100}, \quad  i_2:= \frac{53510733}{2000}, \quad i_3:= \frac{38435553}{51200000}\]
\noindent $>$ \verb"CurvDeg3EllSub"$(i_1,i_2,i_3)$;
 \begin{center} 0 \end{center}
This means that the above curve has a $(3,3)$-split Jacobian.
\end{example}

\begin{example} Let
$y^2=f(x)$ be a genus 2 curve where
$f:=4x^6+(52\sqrt{6}-119)x^5+(39\sqrt{6}-24)x^4+(26\sqrt{6}-54)x^3+(13\sqrt{6}-27)x^2+3x+1$.
Then,\\

\noindent \verb"> a_1:=a_1(f,x); a_2:=a_2(f,x);"

\[
\begin{split}
a_1 &:= \frac{1316599234443}{270840023}\sqrt{6}+\frac{6310855638567}{541680046}, \\
a_2 &:= \frac{-96672521239976}{1183208072032328121}\sqrt{6}+\frac{1467373119039023}{7099248432193968726}
\end{split}
\]
$>$ \verb"CurvDeg3EllSub_J_2"$(a_1,a_2)$
\begin{center} 0 \end{center}
This means that the curve has $J_2=0$ and $(3,3)$-split Jacobian.
\end{example}

\subsection{Rational model of genus 2 curve}
For details on the rational model over its field of moduli see   \cite{Sh7}. The rational
model of $C$ (if such model exists) is determined by the function \verb"Rational_Model()".

\begin{example} Let $y^2=f(x)$ be a genus 2 curve where
$f:=x^5+\sqrt{2}x^3+x$. Then,\\

\noindent $>$ \verb"Rational_Model(f,x)";\\
\[x^5+x^3+\frac{1}{2}x\]
\end{example}

\begin{example} Let $y^2=f(x)$ be a genus 2 curve where
$f:=5x^6+x^4+\sqrt{2}x+1$. Then,\\

\noindent $>$ \verb"Rational_Model(f,x)";\\

\begin{tiny}
\begin{equation*}
\begin{split}
 & -365544026018739971082698131028050365165449396926201478x^6\\
 &-606501618836700589954579317910699990585971018672445125x^5\\
 &-369842283192872727990502041940062429271727924754392250x^4\\
 &-32387676975314893414920003149434215247663074288356250x^3\\
 &+74168490079198328987047652288420271784298171220937500x^2\\
 &+38274648493772601723357350829541971828965732551171875x\\
 &+6501732463119213927460859571034949543087123367187500
 \end{split}
\end{equation*}
\end{tiny}
\end{example}

Notice that our algorithm doesn't always find the minimal rational model of the curve. An efficient way to
do this has yet to be determined.

\subsection{A different set of invariants}

As explained in Section 2, invariants $i_1, i_2, i_3$ were defined that way for computational benefits. However, they make the results involve many subcases and are inconvinient at times. In the second version the the \texttt{genus2} package we intend to convert all the results to the $t_1, t_2, t_3$ invariants
\[ t_1 = \frac {J_2^5} {J_{10}}, \quad  t_2 = \frac {J_4^5} {J_{10}^2}, \quad t_3  = \frac {J_6^5} {J_{10}^3}.\]

The other improvement of version two is that when the moduli point $\p$ is given the equation of the curve is given as the minimal equation over the minimal field of definition.


\begin{thebibliography}{99}


\bibitem {Al} {\sc Al-Shemas, Eman}, Resolvent equations method for general variational inclusions.  Albanian J. Math.  3  (2009),  no. 3, 107--116.

\bibitem {AL}  {\sc Ayad, Mohamed; Luca, Florian},  Fields generated by roots of $x^n+ax+b$.  Albanian J. Math.  3  (2009),  no. 3, 95--105.


\bibitem {BNP}  {\sc Banks, William D.; Nevans, C. Wesley; Pomerance, Carl},  A remark on Giuga's conjecture and Lehmer's totient problem.  Albanian J. Math.  3  (2009),  no. 2, 81--85.

\bibitem {Be1} {\sc L. Beshaj}, Singular locus of the Shaska's surface,    (submitted)

\bibitem {BL} {\sc R. Broker, K. Lauter}, 
Modular polynomials for genus 2.  
LMS J. Comput. Math. 12 (2009), 326–339.

\bibitem {BLP}   Bernard, Nicolas; Lepr´evost, Franck;
Pohst, Michael, Jacobians of genus-2 curves with a rational point of order 11.  Experiment. Math. 18 (2009), no. 1, 65–70.

\bibitem {Be2} {\sc L. Beshaj}, The arithmetic of genus two curves,         (work in progress).


\bibitem {Be5} {\sc L. Beshaj, A. Duka, V. Hoxha, T. Shaska} Computational tools for genus two curves, (work in progress).

\bibitem {Blake} {\sc I. Blake, G. Seroussi and N.  Smart},  Elliptic Curves in Cryptography,  LMS, 265, (1999).

\bibitem {BW} {\sc C. Birkenhake, H. Wilhelm}, Humbert surfaces and the Kummer plane.  Trans. Amer.
Math. Soc.  355  (2003),  no. 5, 1819--1841.

\bibitem{BeBiLaPe08}
D.~J. Bernstein, P.~Birkner, T.~Lange, and C.~Peters, \emph{{ECM} using
  {E}dwards curves}, Cryptology ePrint Archive, 2008,
  http://eprint.iacr.org/2008/016.

\bibitem {Bo} {\sc O. Bolza}, On binary sextics with linear transformations
into themselves. {\it Amer. J. Math.} {\bf 10}, 47-70.


\bibitem {CN} {\sc C. -L. Chai, P. Norman}, {\it Bad reduction of the Siegel moduli scheme of genus two with $\Gamma_0(p)$-level structure}, Amer. J. Math.  {\bf 122}, (1990), 1003-1071.

\bibitem {Cl} {\sc A. Clebsch},
Theorie der Bin\"aren Algebraischen Formen, Verlag von B.G. Teubner, Leipzig, 1872.

\bibitem {Co}  {\sc R. Cosset},   Factorization with genus 2 curves. (preprint)


\bibitem {Du} {\sc R. Dupont}, Moyenne arithmetico-geometrique,suites de Borchardt et applications, {\it J.PhD thesis, Ecole Polytechnique.} {\bf 1}Paris
      (2006)

\bibitem {DS} {\sc A. Duka  and  T. Shaska}  Modular polynomials of genus two, preprint

\bibitem{Duquesne07}
S.~Duquesne, \emph{Improving the arithmetic of elliptic curve in the {J}acobi
  model}, Inform. Process. Lett. \textbf{104} (2007), 101--105.

\bibitem {Du} {\sc I. Duursma  and N. Kiyavash},
The Vector Decomposition Problem for Elliptic and Hyperelliptic Curves, (preprint)


\bibitem {El}   {\sc Elezi, Artur},  Toric fibrations and mirror symmetry.  Albanian J. Math.  1  (2007),  no. 4, 223--233.

\bibitem {EL} {\sc K. Eisentrager, K. Lauter}, A CRT {\it algorithm for constructing genus 2 curves over finite fields},
to appear in Arithmetic, Geometry and Coding Theory (AGCT-10), 2005.

\bibitem {Enge} {\sc A. Enge},  Computing modular polynomials in quasi-linear time.  Math. Comp.  78  (2009),  no. 267, 1809--1824.

\bibitem {EP}  {\sc Elkin, Arsen; Pries, Rachel},  Hyperelliptic curves with $a$-number 1 in small characteristic.  Albanian J. Math.  1  (2007),  no. 4, 245--252.


 


\bibitem {Ga} {\sc Gashi, Qëndrim R.},  A vanishing result for toric varieties associated with root systems.  Albanian J. Math.  1  (2007),  no. 4, 235--244.

\bibitem{Gaudry07}
P.~Gaudry, \emph{Fast genus 2 arithmetic based on theta functions}, J. Math.
  Cryptol. \textbf{1} (2007), 243--265.

\bibitem{GaSc01}
P.~Gaudry and {\'E}.~Schost, \emph{On the invariants of the quotients of the
  {J}acobian of a curve of genus 2}, Applied Algebra, Algebraic Algorithms and
  Error-Correcting Codes (S.~Bozta\c{s} and I.~Shparlinski, eds.), Lecture
  Notes in Comput. Sci., vol. 2227, Springer-Verlag, 2001, pp.~373--386.

\bibitem {GH} {\sc P. Gaudry, R. Harley}, {\it Counting points on hyperelliptic curves over finite fields},Algorithmic Number Theory Symposium IV, Springer Lecture Notes in Computer
    Sience, vol. 1838, 2000, pp. 313-332.

\bibitem {GHKRW} {\sc P. Gaudry, T. Houtman, D. Kohel, C. Ritzenthaler, A. Weng}, {\it The 2-adic CM-method for
  genus 2 curves with applications to cryptography}, Asiacrypt, Springer Lecture Notes in Computer Science, vol. 4284, 2006, pp. 114-129

\bibitem {GaSc} {\sc P. Gaudry, E. Schost}, {\it Modular equations for hyperelliptic curves}, Math, Comp,{\bf 74} vol. (2005),
  429-454.


\bibitem {GSh} \textsc{J. Gutierrez and T. Shaska},
Hyperelliptic curves with extra involutions, \emph{LMS J. of Comput. Math.}, 8 (2005), 102-115.


\bibitem {HJ}  {\sc Haran, D.; Jarden, M.}, Regular lifting of covers over ample fields.  Albanian J. Math.  1  (2007),  no. 4, 179--185.

\bibitem {Hi} {\sc R. Hidalgo}, 
Classical Schottky uniformizations of Genus 2.  A package for MATHEMATICA.  
Sci. Ser. A Math. Sci. (N.S.) 15 (2007), 67–94.

\bibitem {HS} {\sc  V. Hoxha and  T. Shaska}, Factoring large numbers by using genus two curves, (submitted)

\bibitem {Hu} {\sc G. Humbert}
Sur les fonctionnes abéliennes singulières. I, II, III. J. Math. Pures Appl. serie 5, t. V, 233--350 (1899); t. VI, 279--386 (1900); t. VII, 97--123 (1901).


\bibitem {Ig} {\sc J. Igusa}, Arithmetic Variety Moduli for genus 2. \emph{
Ann. of Math}. (2), 72, 612-649, 1960.

\bibitem {Ig2} {\sc J. -I. Igusa}, {\it On Siegel modular forms of genus two}, Amer. J. Math.{\bf 84} (1962), 175-200.

\bibitem {J} {\sc C. Jacobi}, Review of  Legendre,
Th\'eorie des fonctions elliptiques. Troiseme suppl\'em ent. 1832. J. reine angew. Math. 8, 413-417.

\bibitem {Ju} {\sc B. Justus},  On integers with two prime factors.  Albanian J. Math.  3  (2009),  no. 4, 189--197. 

\bibitem {JSB}  {\sc Joswig, Michael; Sturmfels, Bernd; Yu}, Josephine Affine buildings and tropical convexity.  Albanian J. Math.  1  (2007),  no. 4, 187--211.
    
    
\bibitem {JKV} {\sc Joyner, David; Ksir, Amy; Vogeler, Roger},  Group representations on Riemann-Roch spaces of some Hurwitz curves.  Albanian J. Math.  1  (2007),  no. 2, 67--85 (electronic).

\bibitem {Kr} {\sc A. Krazer}, Lehrbuch der Thetafunctionen, Chelsea, New York, 1970.

\bibitem {KSV} \textsc{V. Krishnamorthy, T. Shaska,   H. V\"olklein},
Invariants of binary forms , \emph{Developments in Mathematics}, Vol. 12, Springer 2005, pg. 101-122.

\bibitem {Ko} {\sc Kopeliovich, Yaacov}, Modular equations of order $p$ and theta functions.  Albanian J. Math.  1  (2007),  no. 4, 271--282.


\bibitem{Lenstra87}
H.~W. {Lenstra, Jr.}, \emph{Factoring integers with elliptic curves}, Ann. of  Math. (2) \textbf{126} (1987), 649--673.


\bibitem  {LS} {\sc Luca, Florian; Shparlinski, Igor E.},  Pseudoprimes in certain linear recurrences.  Albanian J. Math.  1  (2007),  no. 3, 125--131 (electronic).

\bibitem {deg5} \textsc{K. Magaard, T. Shaska, H. V\"olklein},    Genus 2 curves with degree 5 elliptic subcovers, \emph{Forum. Math.}, vol. \textbf{16}, 2, pg. 263-280, 2004.


\bibitem {MVG} {\sc Magaard, Kay; Völklein, Helmut; Wiesend, Götz},  The combinatorics of degenerate covers and an application for general curves of genus 3.  Albanian J. Math.  2  (2008),  no. 3, 145--158.

\bibitem  {MS} \textsc{K. Magaard, T. Shaska, S. Shpectorov, and H. V\"olklein},
The locus of curves with prescribed automorphism group. \emph{Communications in arithmetic fundamental
groups} (Kyoto, 1999/2001). S\=urikaisekikenky\=usho K\=oky\=uroku No. 1267 (2002), 112--141.

\bibitem {Me} {\sc J. -F. Mestre}, {\it Construction des curbes de genre 2 a partir de leurs modules}, Effective Methods
  in Algebraic Geometry, Birkhauser, Progress in Mathematics, vol. 94, 1991, pp. 313-334.

 

\bibitem {Mura} {\sc N. Murabayashi}, {\it The moduli space of curves of genus two covering elliptic curves}, Manuscripta
   Math.{\bf 84} (1994), 125-133.



\bibitem {PSW}  {\sc Previato, E,; Shaska, T.; Wijesiri, S.}, Thetanulls of cyclic curves of small genus, Albanian J. Math., vol. 1, Nr. 4, 2007, 265-282.


\bibitem {Sa} {\sc R. Sanjeewa},  Automorphism groups of cyclic curves defined over finite fields of any characteristics.  Albanian J. Math.  3  (2009),  no. 4, 131--160.

\bibitem  {Sh7} \textsc{T. Shaska}, Curves of genus 2 with $(n, n)$-decomposable Jacobians,
\emph{J. Symbolic Comput.} 31 (2001), no. 5, 603--617.

\bibitem {Sh6} \textsc{T. Shaska},
Genus 2 curves with (3,3)-split Jacobian and large automorphism group, Algorithmic Number Theory (Sydney,
2002), \textbf{ 6}, 205-218, \emph{Lect. Not. in Comp. Sci.}, 2369, Springer, Berlin, 2002.

\bibitem {deg3} \textsc{T. Shaska}, Genus 2 curves with degree 3 elliptic subcovers,
\emph{Forum. Math.}, vol. \textbf{16}, 2, pg. 263-280, 2004.

\bibitem {Sh2} \textsc{T. Shaska}, Some special families of hyperelliptic curves,
\emph{J. Algebra Appl.}, vol \textbf{3}, No. 1 (2004), 75-89.

\bibitem {Sh} {\sc T.Shaska}, Genus 2 curves covering elliptic curves, a computational approach {\it Lect.Notes in Comp.} 
{\bf 13} (2005)

\bibitem {SV1} \textsc{T. Shaska and H. V\"olklein}, Elliptic subfields and automorphisms of genus two
fields, \emph{Algebra, Arithmetic and Geometry with Applications}, pg. 687 - 707, Springer (2004).

\bibitem {ShW} \textsc{T. Shaska \and S. Wijesiri},  Theta functions and algebraic curves with automorphisms, \emph{Algebraic Aspects of Digital Communications}, pg. 193-237, NATO Advanced Study Institute, vol. 24,  IOS Press, 2009.


\bibitem{Wamelen98}
P.~van Wamelen, \emph{Equations for the {J}acobian of a hyperelliptic curve},
  Trans. Amer. Math. Soc. \textbf{350} (1998), no.~8, 3083--3106.



\end{thebibliography}
\end{document}